\documentclass[a4paper,11pt]{amsart}

\usepackage{amsmath,amsthm,amssymb,verbatim,amscd,epsfig,esint}
\usepackage[latin1]{inputenc}
\usepackage[T1]{fontenc}
\usepackage{enumerate}

\theoremstyle{plain}
\newtheorem{thm}{Theorem}[section]
\newtheorem{lem}[thm]{Lemma}
\newtheorem{cor}[thm]{Corollary}
\newtheorem{prop}[thm]{Proposition}

\theoremstyle{definition}
\newtheorem{defn}[thm]{Definition}

\newtheorem{rem}[thm]{Remark}

\numberwithin{equation}{section}

\newcommand{\C}{\mathbb C}
\newcommand{\D}{\mathbb D}

\newcommand{\R}{\mathbb R}

\newcommand{\cF}{\mathcal F}

\newcommand{\cS}{\mathcal S}
\newcommand{\cH}{\mathcal H}

\newcommand{\zbar}{\bar{z}}
\newcommand{\wbar}{\bar{w}}

\renewcommand{\Re}{\operatorname{Re}}
\renewcommand{\Im}{\operatorname{Im}}

\newcommand{\loc}{\mathrm{loc}}

\let\swap=\phi
\let\phi=\varphi
\let\varphi=\swap
\let\swap=\epsilon
\let\epsilon=\varepsilon
\let\varepsilon=\swap

\let\swap=\leq
\let\leq=\leqslant
\let\leqslant=\swap
\let\swap=\geq
\let\geq=\geqslant
\let\geqslant=\swap

\title [Manifolds of quasiconformal mappings]
{Manifolds of quasiconformal mappings and\\ the nonlinear Beltrami equation}\date{}

\author{Kari Astala}
\address{LE STUDIUM, Loire Valley Institute for Advanced Studies, Orl\'eans \& Tours, France, MAPMO, rue de Chartres, 45100 Orl\'eans, France; 
Department of Mathematics and Statistics, University of Helsinki, 
         P.O. Box 68, FI-00014, Helsinki, Finland}
\email{kari.astala@helsinki.fi}

\author{Albert Clop}
\address{Department of Mathematics, Universitat Aut\`onoma de Barcelona, 08193 Bellaterra (Barcelona), Catalonia}
\curraddr{}
\email{albertcp@mat.uab.cat}

\author{Daniel Faraco}
\address{Department of Mathematics, Universidad Aut\'onoma de Madrid, 28049 Ma\-drid, 
Spain; ICMAT CSIC-UAM-UCM-UC3M, 28049 Madrid, Spain}
\curraddr{}
\email{daniel.faraco@uam.es}

\author{Jarmo  J\"a\"askel\"ainen}
\address{Department of Mathematics and Statistics, University of Helsinki, 
         P.O. Box 68, FI-00014, Helsinki, Finland; Department of Mathematics, Universidad Aut\'onoma de Madrid, 28049 Madrid, 
Spain}
\curraddr{}
\email{jarmo.jaaskelainen@helsinki.fi}

\thanks{K.A. was supported by Academy of Finland project  SA-12719831. 
A.C. was supported by research grant   2014SGR75 (Generalitat de Catalunya), MTM2013-44699 (Ministerio de Economia y competitividad), Subprograma Ramon y Cajal (Gobierno de Espa\~na) and the EU initial training network MAnET Metric analysis for emergent technologies.  D.F. was supported by research grant MTM2011-28198 from the Ministerio de Ciencia e Innovaci\'on (MCINN), by MINECO: ICMAT Severo Ochoa project SEV-2011-0087, and by the ERC 301179.
J.J. was supported by the ERC 301179 and Academy of Finland (no. 276233)}

\keywords{Quasiconformal mappings, nonlinear Beltrami equation, submanifolds of Fr\'echet manifold}
\subjclass[2010]{Primary {30C62}; Secondary {35J60}, {35J46}}

\begin{document}

\frenchspacing

\begin{abstract}
\noindent In this paper we show  that the homeomorphic solutions to each nonlinear Beltrami equation $\partial_{\zbar}f = \cH(z, \partial_z f)$ generate a two-dimensional manifold  of quasiconformal mappings ${\cF}_{\cH} \subset W^{1,2}_{\loc}(\C)$.  Moreover, we show that under regularity assumptions on $\cH$, the manifold ${\cF}_{\cH}$ defines the structure function $\cH$ uniquely.
\end{abstract}

\maketitle

\section{Introduction}
\noindent In the context  of  $G$-compactness of  Beltrami operators (introduced in \cite{G-cpt}) it was recently proved that there is a one-to-one correspondence between the two-dimensional vector spaces of quasiconformal mappings and the $\R$-linear Beltrami equations  \cite{AN}, \cite{AJ}, \cite[Theorem 16.6.6]{AIM}. More precisely, given a linear family  $\mathcal{F}=\{
\alpha f+ \beta g : \alpha, \beta \in \R, \; \alpha^2 + \beta^2 \neq 0\} \subset W^{1,2}_{\loc}(\C)$ consisting of quasiconformal maps, there exists a unique couple of measurable functions $\mu$ and $\nu$ with $|\mu|+|\nu| 
 \leq k <1$   such that $\mathcal{F}$ is exactly  the set homeomorphic solutions to the $\R$-linear Beltrami equation 
\begin{equation} \label{1}
\partial_{\zbar} f(z)=\mu (z)\, \partial_{z} f(z) +\nu(z)\, \overline{\partial_{z} f(z) }, \qquad \text{for almost every $z\in\C$}.
\end{equation}
If $\nu\equiv0$, one gets the classical Beltrami equation and, in this case, the family is not only $\mathbb{R}$-linear, but also $\mathbb{C}$-linear. 

The goal of this paper is to study the corresponding questions in the nonlinear setting and, in particular, to understand the structure of the set of homeomorphic solutions to the nonlinear 
Beltrami equations
\begin{equation}\label{Hqr}
\partial_{\zbar} f(z)=\cH(z,\partial_{z} f(z)), \qquad \text{for almost every $z\in\C$}.
\end{equation}
 Starting from the pioneering work of Bojarski and Iwaniec \cite{Boj74}, \cite{iw}, \cite{boj} in the last two decades it has been shown that much of the linear theory for the Beltrami equation extends to the nonlinear situation, under basic assumptions in the nonlinearity to guarantee the uniform ellipticity. Namely, we consider  structure functions  $\cH:\C\times\C\to\C$  which satisfy
\begin{enumerate}[(H1)]
\item $\cH$ is $k$-Lipschitz in the second variable, that is, for $w_1, w_2\in\C$,
$$
|\cH(z,w_1)-\cH(z,w_2)|\leq k(z)|w_1-w_2|, \quad 0 \leq k(z) = \frac{K(z) - 1}{K(z) + 1} \leq k < 1,
$$
for almost every $z\in\C$, and the normalization $\cH(z,0)\equiv0$ holds.
\smallskip

\item For every $w\in\C$, the mapping $z\mapsto \cH(z,w)$ is measurable on $\C$.
\end{enumerate}
\medskip

\noindent Starting with these assumptions we  study the corresponding nonlinear Beltrami equations \eqref{Hqr}. The existence and the regularity theories of the nonlinear Beltrami equations resemble those of the linear one, see  \cite{Boj74}, \cite{iw}, \cite{boj}, \cite{AIS}, \cite{AIM}. However, the uniqueness of a homeomorphic solution is much more subtle than for linear equations,
where such solutions are determined by their values at two distinct points. For the nonlinear  equations this fact does not need to remain true,
as proved in \cite{ACFJS}. Therefore we make the following definition.
\begin{defn}[Uniqueness property]\label{uniqueness} 
We say that the Beltrami equation  \eqref{Hqr} has \emph{the uniqueness property} if for every $z_0, \, z_1, \, \omega_0,\, \omega_1 \in \mathbb{C},$ with $  z_0 \neq z_1, w_0 \neq w_1$, there is a unique homeomorphic solution  $f\in W^{1,2}_{\loc}(\C)$ to \eqref{Hqr} such that
$f(z_0)=\omega_0$ and $f(z_1)=\omega_1$.
\end{defn}

Thus all the linear equations~\eqref{1}  have  the uniqueness property, see \cite[Corollary 6.2.4]{AIM}. In the nonlinear case,
the equation \eqref{Hqr} has the uniqueness property if 
\begin{equation}\label{inftybound}
\limsup_{|z|\to\infty} k(z) < 3 - 2\sqrt{2} ,
\end{equation}
 but uniqueness property may fail  when \eqref{inftybound} is not true, see  \cite[Theorem 1.1]{ACFJS}. 
 On the other hand, we have the uniqueness property, e.g.  if $\cH$ is $1$-homogeneous in the second variable, see \cite[Theorem 8.6.2]{AIM} or \cite[Theorem 1.3]{ACFJS}. In the terms of the quasiconformal distortion, the 
  bound \eqref{inftybound} reads as $K(z) < \sqrt{2}$ near the infinity.

In the rest of the paper
we only consider structure functions $\cH$ such that the associated equation has the uniqueness property.
To every such structure function we can associate  a  family of quasiconformal mappings  $\cF_{\cH} = \{\phi_a\}_{a \in \mathbb{C}}$ by setting as $\phi_a$, $a\neq 0$, the unique $W^{1,2}_{\loc}(\C)$-homeomorphic solution to \eqref{Hqr} such that 
$$
\phi_a(0)=0, \qquad  \phi_a(1)=a.
$$
It is convenient to set $\phi_0(z) \equiv 0$.

It follows from the Lipschitz regularity of $\cH$, (H1), that if $f$ and $g$ are arbitrary $W^{1,2}_{\loc}$-solutions to \eqref{Hqr}, then $f-g$ is $K$-quasiregular; indeed,
\begin{equation}\label{Hqrextra}
|\partial_{\zbar}f(z) - \partial_{\zbar}g(z)| = |\cH(z, \partial_z f(z)) - \cH(z, \partial_z g(z))| \leq k\,|\partial_zf(z) - \partial_z g(z)|.
\end{equation}
 Furthermore, it can  be shown,  see Proposition~\ref{differences}, that the above uniqueness property is equivalent to requiring that for  homeomorphic solutions the difference  $f-g$  is either a constant or a homeomorphism. Hence in the latter case the difference is 
 quasiconformal, and this motivates the following definition. 
  
\begin{defn}[Field of quasiconformal mappings]\label{perhe}
We call a set $\cF = \{\phi_a\}_{a\in\,\C} \subset W^{1,2}_{\loc}(\C)$ a {\it field of quasiconformal mappings}, if the following holds for some $1 \leq K < \infty$.
\begin{enumerate}[(F1)]
\item  If $a\neq 0$, then $\phi_a$ is a $K$-quasiconformal mapping of $\C$, with $\phi_a(0)=0$ and $\phi_a(1)=a$. If $a=0$,  $\phi_0\equiv 0$.
\item The difference $\phi_a-\phi_b$ is $K$-quasiconformal, for $a\neq b$.
\end{enumerate}
If the field arises from a structure function $\cH$, as above, we denote it by ${\cF}_{\cH}$.  
  \end{defn}
  
 In Section 2 we prove that if  a field of quasiconformal mappings  $\cF$ is  $C^1$ in the parameter $a$, i.e., $a \mapsto \phi_a(z) \in C^1(\C)$ for every fixed  $z$,
 then the field is a $C^1$-embedded submanifold of the space $L^\infty_{\loc}(\C)$, a surface whose points are quasiconformal mappings. Here we use the setting and the definition of submanifolds of Fr\'echet spaces from Lang \cite{Lang}.
 
  Even without extra smoothness assumptions, for almost every $a$ the  tangent plane, $T_{\phi_a}\cF$, at the point $\phi_a$ exists and equals the set of the
  homeomorphic solutions to an  $\R$-linear Beltrami equation (see Proposition~\ref{dif} and Remark~\ref{bundleref}). 
  
  Moreover, we prove that if the field is induced by a nonlinear Beltrami equation determined by the structure function $\cH$, then $\cF_{\cH} $ is a ($C^1$-embedded) submanifold  of $W^{1,2}_{\loc}(\C)$, with the tangent bundle $T\cF_{\cH}:= \bigcup_{a\neq0} T_{\phi_a}\cF_{\cH}$  given by solutions to the $\R$-linear Beltrami equations that are obtained by linearizing the starting equation \eqref{Hqr}.

 \begin{thm}\label{partialder1}
Assume that $\cH$ has the uniqueness property and let $\cF_\cH=\{\phi_a (z)\}_{a\in\C}$.  If  $w\mapsto \cH(z, w) \in C^1(\C)$, then for every fixed $a \in \C$ 
the directional derivatives 
$$\partial_{e}^a \phi_a(z) := \lim_{ t \to 0^+} \frac{\phi_{a+te}(z) - \phi_a(z)}{t}, \qquad e \in \C,$$ are all quasiconformal mappings of $z$, satisfying the same $\R$-linear  Beltrami equation  
\begin{equation}\label{rlinearaint}
\partial_{\zbar}f(z) = \mu_a(z)\,\partial_z f(z) + \nu_a(z)\,\overline{\partial_z f(z)} \qquad \text{a.e.}
\end{equation}
\end{thm}
\medskip

It turns out that
\begin{equation}\label{rlineara2}
\mu_a(z) = \partial_w\cH\big(z, \partial_z \phi_a(z)\big)\quad \text{ and } \qquad \nu_a(z) = \partial_{\bar{w}}\cH\big(z, \partial_z \phi_a(z)\big),
\end{equation}
and, moreover, $a\mapsto\phi_a$ is a continuously  differentiable $W^{1,2}_{\loc}(\C)$-valued function. 
  \medskip
  
Next, we 
investigate to which extent the relation between the structure function  $\cH$ and the field $\cF_{\cH}$  is unique, as it is when $\cH$ is linear.
Given any field of quasiconformal mappings  $\cF$  as in Definition \ref{perhe} we can formally associate to it  a nonlinear Beltrami equation represented by some structure function $\cH_{\cF}$. Indeed, simply by starting with the necessary condition
\begin{equation}\label{kaava}
\cH_{\cF}(z,w)=\partial_{\bar z} \phi_a(z) \quad \mbox{if } w=\partial_z \phi_a (z),
\end{equation}
and then extending the structure function to the whole $\C \times \C$, for example, by using Kirzsbraun's extension theorem. Note that by the  condition (F2) in Definition \ref{perhe}, the identity \eqref{kaava} gives a well-defined structure function. In general this structure function $\cH_\cF$ does not need to satisfy the conditions (H1) and (H2). 

However, starting from a smooth enough $\cH$, the field determines the structure function uniquely. 
\begin{defn}[Regular structure function]\label{defregular}
We say that the structure function $\cH$  is {\em regular} if, in addition to (H1) and (H2),
$\cH$ has the uniqueness property and  if for some $\alpha \in (0, 1)$,  
\begin{equation}\label{holdercondition}
|\cH(z_1, w) - \cH(z_2, w)| \leq \mathbf{H}_{\alpha}(\Omega)\,|z_1 - z_2|^{\alpha}|w| \qquad z_i \in \Omega, \quad w \in \C,
\end{equation}
for all $\Omega \subset \C$ bounded, and 
\begin{equation*}
(z, w) \mapsto D_w \cH(z,w)  \quad \text{is locally $\alpha$-H\"older continuous in } \C \times \C.
\end{equation*}
\end{defn} 
\medskip

Under these regularity assumptions we obtain
 
\begin{thm}\label{Uniquenessthm}
Suppose that $\cH$ is  regular. Then   $\cF_{\cH}$ defines $\cH$ uniquely,  that is,  ${\cH}_{{\cF}_\cH}=\cH$.
\end{thm}

Notice that since $\cH =\cH_\cF $ satisfies \eqref{kaava},  
 the  properties of $a \mapsto \partial_z \phi_a (z)$  
 play a fundamental role  in proving  the   uniqueness of  $\cH$. Even in the linear case, the fact that  $a\mapsto\partial_z \phi_a (z)$ is a bijection of $\C$  is 
 nontrivial, as it requires the recently proved   Wronsky-type theorems \cite{AN}, \cite{AJ}. Thus an important  part of the proof of Theorem~\ref{Uniquenessthm} is to show the following result.

\begin{thm}\label{homeo} 
Suppose that $\cH$ is regular and  $\cF_\cH=\{\phi_a (z)\}_{a\in\C}$. Then the map $a \mapsto \partial_z \phi_a(z)$ is a homeomorphism on the Riemann sphere, for every fixed $z\in \C$.  \end{thm}

The proof of this statement requires a number of results of 
independent interest in the study of nonlinear Beltrami equations. The gain of regularity in  the $a$-variable is reflected in our Theorem~{\ref{partialder1}}. Concerning the behaviour in the $z$-variable, here we first need  Schauder-type estimates for nonlinear Beltrami equations.

\begin{thm}\label{schauder}
Let the structure function $\cH$ satisfy \eqref{holdercondition}. Then every quasiregular solution $f$ to
$$
\partial_{\zbar} f(z) = \cH(z, \partial_z f(z)) \qquad \text{for a.e. $z\in \Omega$}
$$
belongs to $C^{1, \gamma}_{\loc}(\Omega)$ where $\gamma = \alpha$, if $\alpha<\frac1K$, and otherwise one can take any $\gamma<\frac1K$. Here $K = \frac{1 + k}{1 - k}$. Moreover, we have a norm bound when $\D(z_1, 2r) \Subset \Omega$,
\begin{equation}\label{thmnorm}
{\|D_zf\|}_{C^\gamma(\D(z_1, r))}
\leq c(K, \alpha, \gamma, z_1, r, \mathbf{H}_\alpha(\Omega))\,\|D_z f\|_{L^2(\D(z_1, 2r))}.\end{equation}\end{thm}

The result is of independent interest, but for  smoothness of presentation we have decided to publish the result and its companion, Theorem~\ref{Jac}, in a separate work \cite{ACFKJ}. Let us emphasise that in Theorem~\ref{schauder} we do not assume the uniqueness property.

For regular structure functions, a combination of Theorem~\ref{partialder1} and Theorem~\ref{schauder} shows that the tangent space $T_{\phi_a}\cF_{\cH}$ consists of solutions to  $\R$-linear Beltrami equations whose coefficients \eqref{rlineara2} are H\"older continuous. Thus we can use classical Schauder estimates for linear equations  and  are entitled to freely change the order of differentiation. This enables us to transfer information from the $z$-variable to the $a$-variable and vice versa. 

The next key issue is the non-degeneracy of the field $\cF_{\cH}$. 
\begin{prop}\label{locinj}
Let $\cH$ be a regular  structure function and  $\cF_\cH=\{\phi_a (z)\}_{a\in\C}$. Then,  for every \! $a$ and  $z\in\C$, 
\begin{equation}\label{nullLag}
\det[D_a \partial_z \phi_a(z) ] \neq 0,
\end{equation} and the determinant does not change sign.
\end{prop}
It turns out that the Jacobian \eqref{nullLag} is precisely a null Lagrangian expression $\Im(\partial_z f \, \overline{\partial_z g})$ formed by suitable solutions $f$ and $g$ to an $\R$-linear Beltrami equation.
These expressions are known  to be non-vanishing almost everywhere by  the recent  works \cite{G-cpt}, \cite{G-cl}, \cite{AN}, \cite{AJ}.

In the above setting the map $a \mapsto \partial_z \phi_a(z)$ is locally injective from $\C \to \C$, but these considerations still  do not imply non-degeneracy at $a=\infty$. For this we need to use the topology of the Riemann sphere. It turns out that, under the H\"older regularity of $\cH$ on the $z$-variable, this fact is a corollary of the following result, also independent of the uniqueness property.

\begin{thm}\label{Jac}
Let the structure function $\cH$ satisfy \eqref{holdercondition}.
Then a homeomorphic solution  $f \in W^{1,2}_{\loc}(\C)$ to the nonlinear Beltrami equation
\begin{equation*}
\partial_{\zbar} f(z) = \cH(z, \partial_z f(z)) \qquad \text{for a.e. $z\in \C$}
\end{equation*}
has a positive Jacobian, $J(z, f) > 0$.

Further, if $f:\C \to \C$ is a normalized solution, i.e., $f(0) = 0$ and $f(1) = 1$, we have the lower bounds for the Jacobian,
$$ \inf_{z\in\D(0, R)} J(z, f) \geq c(\cH, R) > 0, \qquad 0 < R < \infty. $$
\end{thm}

In the case when $\cH$ is $\C$- or $\R$-linear this theorem  is well-known, but the proofs do not extend to the nonlinear case, hence a genuine nonlinear argument is needed. We refer to \cite{ACFKJ} for further details.

Thus under regularity of $\cH$, the field ${\cF}_{\cH}$ is {\it non-degenerate} in the sense of Definition~\ref{nondegdef} below. A topological argument (Lemma~\ref{keylemma}) then completes the proof, showing that 
that for such  fields  $a \mapsto \partial_z\phi_a(z)$ is a homeomorphism of $\C$ and, furthermore, that the field defines a unique $\cH = {\cH}_{\cF}$. Moreover, regularities of ${\cH}_{\cF}$ depend on the regularity properties of $\cF$.

\smallskip

Concerning the structure of the paper, in Section~2, we prove the basic properties of quasiconformal fields $\cF$. In Section~3 we study fields ${\cF}_{\cH}$ and, in particular, show Theorem~\ref{partialder1} and the manifold structure of ${\cF}_{\cH}$ modelled on $W^{1,2}_\loc(\C)$. In Section~4, we study the smooth structure functions $\cH$  and obtain the smoothness of ${\cF}_{\cH}$. In Section~5, we show that ${\cF}_{\cH}$ is non-degenerate (e.g., Proposition~\ref{locinj}) and give the topological argument which completes the proofs of Theorems~\ref{Uniquenessthm} and \ref{homeo}. We finish the paper by showing  how
  $\cH_{\cF}$ inherits the regularity of $\cF$.

 At the end of the paper we point out some  clear obstructions to naive generalizations of our results. However, the line of research seems very promising in several directions. For instance,
 it would be interesting to investigate what is the minimal regularity needed for structure functions so that the above results remain valid. For another example, it would be interesting to see how the geometric properties of the manifold $\cF_{\cH}$ will depend on the structure of
 $\cH$.

\smallskip

\noindent {\bf Acknowledgements.} \quad  The authors would like to thank Luis Guijarro for interesting discussions on the differential geometric interpretation of the paper.

\section{Manifolds}

\noindent
In this section, $\cF=\{\phi_a(z)\}_{a\in\C}$ will always be a field of $K$-quasiconformal mappings, satisfying (F1) and (F2) in Definition~\ref{perhe}.  Let
$\eta_K:\R_+ \to \R_+$ denote the modulus  quasisymmetry of quasiconformal maps \cite[Corollary~3.10.4]{AIM}; we can choose $\eta_K(t)= C(K)\max\{ t^{K},t^{1/K}\}$. Then (F2), the uniform quasisymmetry of the differences, 
implies that the field is bi-Lipschitz respect to the parameter $a$.

\begin{prop}\label{bilip}
Given a field $\cF=\{\phi_a(z)\}_{a\in\C}$,  $2\leq p<\frac{2K}{K-1}$ and $R>0$, it holds that 
\begin{enumerate}
\item[(a)] $a\mapsto\phi_a(z)$ is bi-Lipschitz, for every $z\in\C$. 
In fact,
$$
\frac{1}{\eta_K(1/|z|)} \leq \frac{|\phi_a(z) - \phi_b(z)|}{|a-b|} \leq \eta_K(|z|) , \quad a\neq b.
$$
\item[(b)] $a\mapsto \phi_a : \C \to L^\infty(\D(0,R))$ is bi-Lipschitz 
$$\frac{1}{\eta_K(1/R)}\leq \frac{\|\phi_a-\phi_b\|_{L^\infty(\D(0,R))}}{|a-b|}\leq \eta_K(R), \quad a\neq b.$$
\item[(c)] $a\mapsto D_z\phi_a : \C \to L^p(\D(0,R))$ is bi-Lipschitz and
$$
 \frac{c(p, K)}{\eta_K(1/R)\,R^{1-\frac2p}}\leq \frac{\|D_z\phi_a-D_z\phi_b\|_{L^p(\D(0, R))}}{|a-b|} \leq   \frac{c(p, K)\,\eta_K(R) }{R^{1-\frac2p}},\quad a\neq b.
$$
\end{enumerate}
\end{prop}

\begin{proof}
 For $a\neq b$, the map $g = \phi_a - \phi_b : \C \to \C$ is $K$-quasiconformal, by assumption (F2). Thus by $\eta_K$-quasisymmetry,
\begin{align*}
\frac{1}{\eta_K(1/|z|)}|a - b| &= \frac{1}{\eta_K(1/|z|)}|g(1) - g(0)|\leq |g(z) - g(0)| = |\phi_a(z) - \phi_b(z)|\\
&\leq \eta_K(|z|)|g(1) - g(0)| = \eta_K(|z|)|a - b|.
\end{align*}
Hence we have shown $(a)$. The statement $(b)$ is immediate.

\medskip

It remains to prove that the bi-Lipschitz property is inherited to the 
derivatives. 
By the $\eta_K$-quasisymmetry or directly from claim $(a)$,
$$\aligned\frac{\pi}{\eta_K(1/R)^2}|a - b|^2&\leq\pi\,\inf_{|z|=R}|(\phi_a - \phi_b)(z)|^2\leq |(\phi_a - \phi_b)(\D(0, R))|\\
&\leq \pi\,\sup_{|z|=R}|(\phi_a - \phi_b)(z)|^2\leq \pi\,\eta_K(R)^2|a - b|^2.\endaligned$$
In addition the $K$-quasiconformal map $\phi_a - \phi_b$  satisfies 
$$
|(\phi_a - \phi_b)(\D(0, R))| = \int_{\D(0, R)} J(z, \phi_a - \phi_b)\, dA(z),
$$
$$\aligned
\frac{1}{K}\int_{\D(0, R)} |D_z(\phi_a(z) - \phi_b(z))|^2 \,dA(z) &\leq \int_{\D(0, R)} J(z, \phi_a - \phi_b) \,dA(z)\\
&\leq \int_{\D(0, R)} |D_z(\phi_a(z) - \phi_b(z))|^2 \,dA(z),\endaligned
$$
and, see e.g. \cite[Corollary 13.2.4]{AIM}, for any $2<p<\frac{2K}{K-1}$
$$
\left(\frac{1}{|\D(0, R)|}\int_{\D(0, R)} |D_z(\phi_a - \phi_b)|^2\right)^{\frac12} \simeq \left(\frac{1}{|\D(0, R)|}\int_{\D(0, R)} |D_z(\phi_a - \phi_b)|^p\right)^{\frac1p} 
$$
with constants that depend only on $K$ and $p$. Combining these estimates one gets $(c)$.
\end{proof}

Next we prove that for almost every $a \in \C$,  at the point $\phi_a$ the field $\cF$ has a  tangent space $T_{\phi_a}\cF \subset L^\infty_{\loc}(\C)$. It turns out that $T_{\phi_a}\cF$ is a two-dimensional linear field of quasiconformal mappings. In particular, $T_{\phi_a}\cF \setminus \{0\}$ is the set of homeomorphic solutions to a linear Beltrami equation 
$$
\partial_{\zbar}f(z) = \mu_a(z)\,\partial_z f(z) + \nu_a(z)\,\overline{\partial_z f(z)} \qquad \text{a.e.}
$$
The explicit form of $\mu_a$, $\nu_a$ is given in \eqref{munu} below. 

Moreover, if the quasiconformal field $\cF$ is $C^1$-continuous with respect to the field parameter $a$ (i.e., for each fixed  $z \in \C$, the map $a \mapsto \phi_a(z)$ is $C^1$),
then  we show that $\mathcal{F}$  is a  $C^1$-embedded submanifold  of $L^\infty_{\loc}(\C)$.

We start with 

\begin{prop}\label{dif}
Assume that $\cF = \{ \phi_a(z) \}_{a\in\C}$ is a field of $K$-quasiconfor\-mal mappings.


Then for a.e. $a \in \C$ there exist $\mu_a,\nu_a\in L^\infty(\C)$ with $\||\mu_a|+|\nu_a|\|_\infty\leq \frac{K-1}{K+1} $, such that for every $e\in\C\setminus\{0\}$  the directional derivative
\begin{equation}\label{extra11}
\partial^a_e\phi_a :=\lim_{t\to 0}\frac{\phi_{a+te}-\phi_a}{t}
\end{equation}
exists and 
is the unique $K$-quasiconformal solution to the problem
$$
\begin{cases}
\partial_{\zbar}f(z) = \mu_a(z)\,\partial_z f(z) + \nu_a(z)\,\overline{\partial_zf(z)} \qquad \text{a.e.}\\
f(0)=0,\quad f(1)=e.
\end{cases}$$
The convergence in \eqref{extra11} is taken in $L^\infty_{\loc}(\C)$.
\end{prop}

\begin{proof} We begin by recalling a Banach space version of the Rademacher theorem, see e.g. \cite[Propositions 4.3 and  6.41]{BL}: If $\Phi: \C \to F$ is locally Lipschitz where $F$ is a Banach space with the Radon-Nikodym property, then $\Phi$ is Fr\'echet differentiable at almost every point $a \in \C$, i.e. there exists a bounded linear map $T_a:\C \to F$ such that
\begin{equation*}
\| \Phi(a+e) - \Phi(a) -T_a(e) \| = \|e\|\,o(\|e\|) \qquad \text{for $e \in \C$.}
\end{equation*}
Furthermore, every reflexive Banach space has the Radon-Nikodym property \cite[Corollary 5.12]{BL}.

We apply these facts to the map   $\Phi: a \mapsto \phi_a( \cdot )$, but to make use of the Radon-Nikodym property let us consider it as having values in $L^2(\D(0,R))$.
Then from Proposition \ref{bilip} $(b)$ we see that there is a set $E \subset \C$ with complement of zero measure such that for  every $a \in E$,
\begin{equation}\label{extra13} 
\lim_{t\to 0^+} \left\| \frac{1}{t} [\phi_{a+t e} -\phi_a] - T_a(e) \right\|_{L^2(\D(0,R))} = 0, \qquad e \in \C.
\end{equation}

On the other hand, as a function of $z$ the maps $\frac{1}{t} [\phi_{a+t e}(z)-\phi_a(z)]$ are, by our assumptions, all $K$-quasiconformal in $\C$,  sending $0$ to $0$ and $1$ to $e$. Therefore by the Montel-type theorem \cite[Theorem 3.9.4]{AIM}, we can take a subsequence $t_j \to 0$ such that $\frac{1}{t_j} [\phi_{a+t e}(z)-\phi_a(z)] \to \eta(z)$ locally uniformly in $\C$, where $\eta:\C \to \C$ is $K$-quasiconformal.  Comparing with \eqref{extra13}  we see that 
$$ \eta(z) = T_a(e)(z) \qquad {\rm almost \; everywhere}. 
$$ 
But then the limit map $\eta$ does not depend on the subsequence $\{ t_j\}$ chosen, and furthermore,  exhausting the plane with countably many disks $\D(0,R)$ we obtain that the (locally uniform in $z$) limits 
\begin{equation}\label{extra14} 
 \partial^a_e\phi_a :=\lim_{t\to 0}\frac{\phi_{a+te}-\phi_a}{t} =  T_a(e)
\end{equation}
are all $K$-quasiconformal maps of $\C$, fixing $0$ and sending $1$ to $e$, and depending {\it linearly} on the parameter $e \in \C \setminus\{0\}$.
Moreover, if $e_1, e_2\in\C$ are $\R$-linearly independent then also $\partial_{e_1}^a\phi_a, \partial_{e_2}^a\phi_a$ are $\R$-linearly independent.
 
We can now invoke
\cite[Theorem 16.6.6]{AIM} to see that for almost every $a$, i.e. for $a \in E$, there exists a unique pair of Beltrami coefficients $\mu_a$ and $\nu_a$ such that $\||\mu_a|+|\nu_a|\|_\infty\leq \frac{K-1}{K+1}$ and the equation
$$\partial_{\zbar}f(z) = \mu_a(z)\,\partial_z f(z) + \nu_a(z)\,\overline{\partial_zf(z)} \qquad \text{a.e.}$$
is satisfied by every member of the family $\{\partial^a_e\phi_a\}_{e\in\C}$. After choosing as generators $e_1=1, e_2=i$, the coefficients $\mu_a$, $\nu_a$ may be precisely described, see \cite[(16.190)]{AIM}, as
\begin{equation}\label{munu}\aligned
\mu_a(z)&=\frac{
\partial_{\zbar}(\partial_1^a\phi_a)\,\overline{\partial_z (\partial_i^a\phi_a)}-\partial_{\zbar}(\partial_i^a\phi_a)\,\overline{\partial_z (\partial_1^a\phi_a)}}{2i\,\Im(\partial_z(\partial_1^a\phi_a)\,\overline{\partial_z(\partial_i^a\phi_a)})},\\
\nu_a(z)&=\frac{\partial_z (\partial_1^a\phi_a)\,\partial_{\zbar}(\partial_i^a\phi_a)-\partial_z (\partial_i^a\phi_a)\,\partial_{\zbar}(\partial_1^a\phi_a)}{2i\,\Im(\partial_z(\partial_1^a\phi_a)\,\overline{\partial_z(\partial_i^a\phi_a)})}.
\endaligned
\end{equation}

\end{proof}
 
\begin{rem} We will later in Theorem \ref{existDaFa} see that if  $\cF = \{ \phi_a(z) \}_{a\in\C}$ arises as the set of homeomorphic solutions to a nonlinear Beltrami equation with structure function $\cH$, then $$\mu_a(z) = \partial_w\cH\big(z, \partial_z \phi_a(z)\big)\quad \text{ and } \quad \nu_a(z) = \partial_{\bar{w}}\cH\big(z, \partial_z \phi_a(z)\big).$$
Thus for a general field $\cF$, \eqref{munu} seems to determine a structure function at least infinitesimally. One may then ask if this can be made global, e.g. whether there is a counterpart for the Frobenius theorem in the setting of quasiconformal fields. 
\end{rem}

In addition,  pointwise smoothness of $a \mapsto \phi_a(z)$ 
 forces a uniform (over $z$) continuity of its derivatives.

\begin{cor}\label{corcontdif}
Let $\cF = \{ \phi_a(z) \}_{a\in\C}$ be  a  field of $K$-quasiconformal mappings. Then  $a \mapsto \phi_a$ is Fr\'echet differentiable at almost every $a \in \C$, as  a map  $\C \to L^{\infty}_\loc(\C)$. 

Moreover, if   for each fixed $z\in\C$, $a \mapsto \phi_a(z)$  is $C^1$, then $a \mapsto \phi_a$ is continuously Fr\'echet differentiable  $\C \to L^{\infty}_\loc(\C)$ at every $a \in \C$.
\end{cor}

\begin{proof} Since \eqref{extra14} implies that $\partial_e^a\phi_a(z)$ depends linearly on $e$, the convergence in \eqref{extra14} shows (by definition) that 
\begin{equation*}
a \mapsto \phi_a : \C \to L^{\infty}_{\loc}(\C)
\end{equation*}
is Gateaux differentiable almost everywhere in the $a$-variable. Further, since $a \mapsto \phi_a$ is also bi-Lipschitz by Proposition~\ref{bilip} $(b)$, the Fr\'echet differentiability follows from \cite[p. 84]{BL}.

Assume finally that  for each fixed $z\in\C$, $a \mapsto \phi_a(z)$  is $C^1$.  Then the limit 
$$  \partial^a_e\phi_a(z) :=\lim_{t\to 0}\frac{\phi_{a+te}(z) -\phi_a(z)}{t}
$$
exists for every $a, z \in \C$ and for every $e \in \C \setminus\{0\}$.
Thus the  pointwise differentials $D_a\phi_a(z)\,h  := \partial_1^a\phi_a(z)\,h_1 + \partial_i^a\phi_a(z)\,ih_2$ for $h = h_1 + ih_2 \in \C$ now exist for every $a$, and it remains to show the continuity of  $a \mapsto D_a\phi_a : \C \to L(\C,  L^\infty_\loc(\C))$. 
 Let, for $e = 1, i$, 
$$ g_a(z) =  \partial_{e}^a\phi_a(z). 
$$
Then each $g_a$ is a $K$-quasiconformal map of $\C$ with $g_a(0) = 0$ and $ g_a(1) = e$. That $a \mapsto \phi_a(z)$  is $C^1$ means precisely that 
$$ \lim_{b \to a} g_b(z) = g_a(z) \qquad {\rm pointwise \; in } \; \C. 
$$
But if a sequence of $K$-quasiconformal maps converges pointwise, then \cite[Corollary 3.9.3]{AIM} it converges locally uniformly. In other words, 
\begin{equation*}
\lim_{b \to a}\|  \partial_{e}^a\phi_a -  \partial_{e}^a\phi_b \|_{L^\infty(\D(0, R))}= 0, \qquad 0 < R < \infty.
\end{equation*}
\end{proof}

\begin{prop}\label{mani} 
Assume that $\cF = \{ \phi_a(z) \}_{a\in\C}$ is a field of $K$-quasiconfor\-mal mappings that is $C^1$-continuous with respect to  $a$, i.e., $a \mapsto \phi_a(z) \in C^1(\C)$ for every fixed  $z \in \C$.  Then  $\mathcal{F}$  is a  $C^1$-embedded submanifold  of $L^\infty_{\loc}(\C)$.
\end{prop}

\begin{proof}

We follow the setup and  definitions of submanifolds of Fr\'echet spaces from  \cite[Chapter II]{Lang}. According to this,  we need to show that   
 $a \mapsto \phi_a$ is a topological embedding and an $C^1$-immersion. In fact, the first claim that the map is 
  a homeomorphism onto its image follows from Proposition~\ref{bilip} $(b)$, while Corollary~\ref{corcontdif} shows that
$a \mapsto \phi_a \in C^1(\C, L^\infty_{\loc}(\C))$.

As for the immersion one needs to show \cite[Proposition~2.3, p. 29]{Lang}  
 that the differential $D_a\phi_a$ is injective and it splits  (see  \cite[p. 18]{Lang} for a definition of splitting in this setting). 
 But by \eqref{extra14}, for $e \neq 0$, the image  $D_a\phi_a \, e = \partial^a_{e}\phi_a$ 
 is a quasiconformal mapping sending $0 \mapsto 0$ and $1 \mapsto e$. Hence the kernel of $D_a\phi_a$ is $\{ 0 \}$, which shows the injectivity of the differential $D_a\phi_a: T_a\C \to T_{\phi_a} L^\infty_{\loc}(\C)$.
Since $D_a\phi_a : \C \to L^\infty_\loc(\C)$ is isomorphism onto its image, its (two dimensional) range is complemented and $D_a\phi_a$ splits.
\end{proof}

\begin{rem}\label{bundleref}
Combining Proposition~\ref{dif} with Proposition~\ref{mani}, we see that the tangent space $T_{\phi_a}\cF$ at a given point $\phi_a$ is a two-dimensional field of quasiconformal mappings and they span the two-dimensional field of homeomorphic solutions to a linear Beltrami equation 
$$
\partial_{\zbar}f(z) = \mu_a(z)\,\partial_z f(z) + \nu_a(z)\,\overline{\partial_z f(z)} \qquad \text{a.e.}
$$
\end{rem}

\section{$\cH$-equations}

\begin{prop}\label{differences}
The Beltrami equation  \eqref{Hqr} has the uniqueness property if and only if for every pair of quasiconformal solutions $f$ and $g$ to  \eqref{Hqr} the difference $f-g$ is either quasiconformal or constant.
\end{prop}
\begin{proof} Given two  homeomorphic solutions  $f$ and $g$ to \eqref{Hqr} the difference $f-g$ is quasiregular, see \eqref{Hqrextra}. 

Suppose that \eqref{Hqr} has the uniqueness property and $f-g$ is not injective. Then
 there exist $z_0, z_1 \in \C$ with $(f-g)(z_0) = (f - g)(z_1) =: \gamma$, so that $f$ and $g + \gamma$ are homeomorphic solutions which take the same values at $z_0$ and $z_1$. Thus by the uniqueness property, $f-g$ is constant.

Conversely, assume that the quasiconformal solutions $f$, $g$ to \eqref{Hqr} attain the same values at $z_0$ and $z_1$. 
Then $(f-g)(z_0)= (f-g)(z_1)=0$ so that $f-g$ is not homeomorphic. Under the condition of our Proposition $f-g$ is thus constant, and so $f \equiv g$.
\end{proof}

\begin{rem}
The uniqueness property asks that no two points are in a special role, unlike in the uniqueness of the normalized homeomorphic solution (mapping 0 to 0 and 1 to 1). The uniqueness property is equivalent with the fact that the uniqueness of the normalized solution to the Beltrami equation \eqref{Hqr} is preserved under the pre- and the post-composition with similarities. Actually, it follows from the counterexample in \cite{ACFJS} that the uniqueness of the normalized solution is a truly weaker feature than the uniqueness property.
\end{rem}

If a field $\cF_\cH=\{\phi_a\}_{a\in\C}$ of $K$-quasiconformal maps arises as solutions to an equation \eqref{Hqr}, with the corresponding structure function $\cH$, it is natural to describe the tangent planes $T_{\phi_a}\cF$ in terms of $\cH$. For this we need to require some extra smoothness of $\cH$, namely the continuous differentiability in the gradient variable. 


\begin{thm}[Theorem~\ref{partialder1}]\label{existDaFa}
Assume that $\cH$ has the uniqueness property. Let  $w \mapsto \cH(z, w) \in C^1(\C)$ for every fixed $z\in \C$ and   $\cF_\cH=\{\phi_a (z)\}_{a\in\C}$.  Then for every $a \in \C$ the directional derivatives 
$$\partial_{e}^a \phi_a(z) := \lim_{ t \to 0^+} \frac{\phi_{a+te}(z) - \phi_a(z)}{t}, \qquad e \in \C\setminus \{0\},$$ 
exist and  define quasiconformal mappings of $z$, all  satisfying the same $\R$-linear  Beltrami equation  
\begin{equation}\label{rlineara}
\partial_{\zbar}f(z) = \mu_a(z)\,\partial_z f(z) + \nu_a(z)\,\overline{\partial_z f(z)} \qquad \text{a.e.}
\end{equation}
where 
\begin{equation*}
\mu_a(z) = \partial_w\cH\big(z, \partial_z \phi_a(z)\big)\quad \text{ and } \quad \nu_a(z) = \partial_{\bar{w}}\cH\big(z, \partial_z \phi_a(z)\big).
\end{equation*}
Furthermore, $\partial_{e}^a \phi_a$ is the unique quasiconformal solution to \eqref{rlineara} that fixes $0$ and maps $1$ to $e$.
\end{thm}

\begin{proof} 
Let us fix $a\in\C$, and a direction $e\in\C\setminus\{0\}$, and denote 
$$
\eta^e_t=\frac{\phi_{a+ te}-\phi_a}{t}, \qquad t \in (0, \infty).
$$ 
Since $\cH$ has the uniqueness property, the mappings $\eta^e_t$ are $K$-quasiconformal in $\C$, and they map $0$ to $0$ and $1$ to $e$. Thus, the limit
$$
\eta^e=\lim_{t_j\to 0^+}\eta^ e_{t_j}
$$
exists, at least for a subsequence $t_j$, and it is a $K$-quasiconformal homeomorphism mapping $0$ to itself and $1$ to $e$, see the Montel-type theorem \cite[Theorem 3.9.4]{AIM}.  \label{kqcfamily} Moreover, the limit is taken locally uniformly in $z$, so that also the derivatives converge weakly in $L^p_{\loc}(\C)$ for $p\in \left[2, \frac{2K}{K-1}\right)$. We will show that $\eta^e$ solves an $\R$-linear Beltrami equation with coefficients given by $\mu_a$ and $\nu_a$, and deduce that for any subsequence $t_j$ the limit mapping $\eta^e$ must be the same.

\medskip

We fix the converging subsequence $t_j$. As quasiconformal mappings, $\phi_a$ and $\eta^e_{t_j}$ are differentiable almost everywhere. Since $t_j$ is countable, we have a set of full measure $E$ such that the derivatives of $\phi_a$ and $\eta^e_{t_j}$ exist at any point $z \in E$ and are nonzero (let us remind that quasiconformal mappings have a non-vanishing $\partial_z$-derivative almost everywhere).

Now, fix one such point $z\in E$. By assumption, $w \mapsto \cH(z, w)$ is $C^1$, so the complex partial derivatives $\partial_w \cH(z,\cdot)$, $\partial_{\bar{w}}\cH(z,\cdot)$ are defined at the point $w_0=\partial_z \phi_a(z)$. Hence we can write
\begin{equation}\label{nonhomogeneousequation}
\aligned
\partial_{\zbar}\eta^e_{t_j}(z) &= \frac{\cH\big(z,\partial_z \phi_{a+ t_j e}(z)\big)-\cH\big(z,\partial_z \phi_a(z)\big)}{t_j}\\
&=\frac{\cH\big(z,w_0+t_j\,\partial_z \eta^e_{t_j}(z)\big)-\cH\big(z,w_0\big)}{t_j}\\
&=\mu_a(z)\,\partial_z \eta^e_{t_j}(z)+\nu_a(z)\,\overline{\partial_z \eta^e_{t_j}(z)} + h_{t_j}(z),
\endaligned
\end{equation}
where 
\begin{equation}\label{muanua}
\aligned
\mu_a(z) &= \partial_w\cH\big(z, w_0\big)= \partial_w\cH\big(z, \partial_z \phi_a(z)\big),\\
\nu_a(z) &=  \partial_{\bar{w}}\cH\big(z, w_0\big)=\partial_{\bar{w}}\cH\big(z, \partial_z \phi_a(z)\big),
\endaligned
\end{equation}
and
\begin{equation}\label{extra15}
\aligned
h_{t_j}(z) &= \frac{\cH\big(z,w_0+t_j\,\partial_z \eta^e_{t_j}(z)\big) - \cH\big(z, w_0\big)}{t_j}\\
&\quad-\frac{\partial_w \cH\big(z,w_0\big)\,t_j\,\partial_z \eta^e_{t_j}(z)+\partial_{\bar{w}}\cH\big(z,w_0\big)\,\overline{t_j\,\partial_z \eta^e_{t_j}(z)}}{t_j}.
\endaligned
\end{equation}

The coefficients $\mu_a$ and $\nu_a$ define a genuine Beltrami equation, since they have an ellipticity bound and they are measurable. Indeed, since $w \mapsto \cH(z, w)$ is $k(z)$-Lipschitz by (H1), we get
\begin{equation*}\label{ellipticityD}
|\mu_a(z)| + |\nu_a(z)| \leq |\partial_w\cH(z, w_0)| + |\partial_{\bar{w}}\cH(z, w_0)| = |D_w \cH(z, w_0)| \leq k(z).
\end{equation*}
Measurability follows as $z \mapsto D_w\cH(z, w)$ is measurable as a limit of measurable functions, and $w \mapsto D_w\cH(z, w)$ is continuous by assumption. Thus $(z, w) \mapsto D_w\cH(z, w)$ is jointly measurable as a Carath\'eodory function. Since $w_0=\partial_z \phi_a(z)$ is a measurable function of $z$, the measurability of $\mu_a$ and $\nu_a$ follows.

Now, $\eta^e$ solves the $\R$-linear Beltrami equation \eqref{rlineara} defined by $\mu_a$ and $\nu_a$, if for every compactly supported test function $\xi\in C^{\infty}_{0}(\C)$
\begin{equation}\label{distbelt}
\int_{\C} \xi(z)\left(\partial_{\zbar} \eta^e(z) - \mu_a(z)\, \partial_z \eta^e(z) - \nu_a(z)\, \overline{\partial_z \eta^e(z)}\right) dA(z) = 0.
\end{equation}
We have, by equation~\eqref{nonhomogeneousequation},
\begin{equation}\label{distbelt2}
\aligned
&\left| \int_{\C} \xi(z)\,\left(\partial_{\zbar} \eta^e(z) - \mu_a(z)\, \partial_z \eta^e(z) - \nu_a(z)\, \overline{\partial_z \eta^e(z)}\right)\,dA(z)\right|\\
&\quad \leq  \left|\int_{\C} \xi(z)\left(\partial_{\zbar}-  \mu_a(z)\partial_z-\nu_a(z)\, \overline{\partial_z}\right)(\eta^e(z)-\eta^e_{t_j}(z))\,dA(z)\right| \\ 
&\qquad+ \left| \int_{\C} \xi(z)\, h_{t_j}(z)\,dA(z)\right|.
\endaligned
\end{equation}
Since $\eta^e_{t_j} \to \eta^e$ locally uniformly, $D_z \eta^e_{t_j} \to D_z \eta^e$ weakly. Moreover, by Proposition~\ref{bilip} $(c)$,  it follows that 
\begin{equation}\label{weakbound2}
\|D_z \eta^e_{t_j}\|_{L^p(\D(0, R))} \leq c(p, K, R)
\end{equation}
and hence the convergence of $D_z \eta^e_{t_j}$  is weak in $L^p(\D(0,R))$. Let $R>0$ be such that $\textrm{supp}(\xi) \subset \D(0,R)$. Since  both $\xi\,\mu_a$ and $\xi\,\nu_a$ belong to $L^{\infty}(\C)$, they are suitable  test functions for the weak $L^p(\D(0,R))$-convergence. Thus the first term in the right hand side of \eqref{distbelt2} converges to $0$ and we are left to show that $h_{t_j} \to 0$  weakly.

We proceed as follows. Set
$$
B(R)=\bigcup_{l=1}^\infty \left(\bigcap_{k=1}^\infty\bigcup_{j=k}^\infty \left\{z\in \D(0,R): t_j\,|\partial_z\eta^e_{t_j}(z)|>2^{-l}\right\}\right)\!.
$$
We choose a sparse enough subsequence of $t_j$ (let us denote it $t_j$, too) such that $\lim_{k\to\infty}\sum_{j\geq k}t_j^p=0$. Now, for any fixed $l=1,2,3...$, one has
$$\aligned
\left|\left\{z\in \D(0,R): t_j\,|\partial_z\eta^e_{t_j}(z)|>2^{-l}\right\}\right|
&\leq \,2^{lp}\,t_j^p\int_{\D(0,R)}|\partial_z\eta^e_{t_j}(z)|^p\,dA(z)\\
&\leq 2^{lp}\,c(p,K,R)\, t_j^p,
\endaligned$$
where \eqref{weakbound2} was used at the last step. Then $|B(R)|=0$ easily follows. 

If we now fix a point $z\in (\D(0,R)\cap E)\setminus B(R)$, then by construction one sees that
$$\lim_{j\to\infty} t_j\,|\partial_z\eta^e_{t_j}(z)|=0.$$
Thus, using the $C^1$-smoothness of $w\mapsto\cH(z,w)$ one gets at these points $z$ that
\begin{equation}\label{aeconv}
\lim_{j\to\infty}\frac{|\cH(z,w_j)-\cH(z,w_0)-D_w\cH(z,w_0)\,(w_j-w_0)|}{|w_j-w_0|}=0,
\end{equation}
where we write $w_0=\partial_z\phi_a(z)$ as before, and $w_j=w_0 + t_j\,\partial_z\eta^e_{t_j}(z).$
Summarizing, if $p'=\frac{p}{p-1}$, then
$$\aligned
&\left|\int_{\C} \xi(z)\,h_{t_j}(z)\,dA(z)\right|
=\left|\int_{\D(0,R)} \xi(z)\,\frac{h_{t_j}(z)}{\partial_z\eta^e_{t_j}(z)}\,\partial_z\eta^e_{t_j}(z)\,dA(z)\right|\\
&\quad\leq \left(\int_{\D(0,R)} \left| \xi(z)\,\frac{h_{t_j}(z)}{\partial_z\eta^e_{t_j}(z)}\right|^{p'}\,dA(z)\right)^\frac{1}{p'}\,\left(\int_{\D(0,R)} |\partial_z\eta^e_{t_j}(z)|^p\,dA(z)\right)^\frac1p \!.
\endaligned$$
The second integral above is bounded independently of $j$, by \eqref{weakbound2}. Concerning the first integral, we note that
\begin{equation*}
\frac{|h_{t_j}(z)|}{|\partial_z\eta^e_{t_j}(z)|}=\frac{|\cH(z,w_j)-\cH(z,w_0)-D_w\cH(z,w_0)\,(w_j-w_0)|}{|w_j-w_0|}\leq 2k.
\end{equation*}
Moreover, we know by \eqref{extra15} and \eqref{aeconv} that 
\begin{equation*}
\frac{|h_{t_j}(z)|}{|\partial_z\eta^e_{t_j}(z)|}\to 0, \quad j \to \infty, \qquad \text{a.e. on }\D(0,R).
\end{equation*}
Since $\xi$ is compactly supported and bounded, the dominated convergence theorem then gives  
$$\lim_{j\to\infty}\int_{\C} \xi(z)\,h_{t_j}(z)\,dA(z)=0$$
as desired. 

The above argument gives, in fact, little more, and  we record this for later purposes. Namely, since $\frac{|h_{t_j}(z)|}{|\partial_z\eta^e_{t_j}(z)|}  <2$ is uniformly bounded, the dominated convergence applies to any $L^q$ norm so that, as long as $s < p < \frac{2K}{K-1}$,
\begin{equation}\label{extra16}
\aligned
\int_{\D(0,R)}|h_{t_j}|^s 
&= \int_{\D(0,R)}\left|\frac{h_{t_j}}{\partial_z \eta^e_{t_j}}\right|^s\,|\partial_z \eta^e_{t_j}|^s\\
&\leq \left(\int_{\D(0,R)}\left|\frac{h_{t_j}}{\partial_z \eta^e_{t_j}}\right|^\frac{sp}{p-s}\right)^\frac{p-s}{p}\,\left(\int_{\D(0,R)}|\partial_z\eta^e_{t_j}|^p\right)^\frac{s}{p} \to 0
\endaligned
\end{equation}
as $j\to \infty$.

In any case, we have shown that  \eqref{distbelt} holds, and thus the limit $\eta^e$ of any converging sequence $\eta^e_{t_j}$ solves the $\R$-linear Beltrami equation \eqref{rlineara} with coefficients $\mu_a$ and $\nu_a$. 

If  we are now given another converging subsequence $\eta^e_{\tilde{t}_j}$, then by the above argumentation the limit $\tilde{\eta}^e$ solves the same $\R$-linear Beltrami equation \eqref{rlineara} as $\eta$. But $\R$-linear equations have only one $K$-quasiconformal solution that fixes $0$ and maps $1$ to a given nonzero point, \cite[Theorem 6.2.3]{AIM}. Hence $\eta^e \equiv \tilde{\eta}^e$.
In particular, the full sequence $\eta^e_t$ converges, showing that for every $a \in \C$ the directional derivatives $\partial^a_e \phi_a = \lim_{t\to 0}\eta^e_t$ exist (convergence is  locally uniform) and satisfy \eqref{rlineara}.
\end{proof}

\begin{rem}\label{extra17}
The uniqueness of the homeomorphic solutions to  \eqref{rlineara} implies that the partial derivatives of the mapping $a\mapsto\phi_a(z)$ depend $\R$-linearly on $e$ so that 
\begin{equation*}
\partial^a_e\phi_a(z) = (\Re e)\,\partial^a_1\phi_a(z) + (\Im e)\,\partial^a_i\phi_a(z).
\end{equation*}
 Thus these form 
 a linear field of $K$-quasiconformal mappings indexed by the direction $e$ of the differentiation. 
 \end{rem}
 
 In addition, the directional derivatives $\partial^a_e \phi_a$ are not only pointwise derivatives, but metric derivatives with the $W^{1,2}_{\loc}$-topology.

\begin{cor}\label{FrechetW12}
Assume that $\cH$ has the uniqueness property and suppose $w \mapsto \cH(z, w) \in C^1(\C)$ for every fixed $z\in \C$. Then  $\partial^a_e\phi_a$ is the $e$-directional derivative of 
$$\aligned
a \mapsto\phi_a &: \C \to L^\infty_{\loc}(\C),\\
a \mapsto\phi_a &: \C \to W^{1,2}_{\loc}(\C).
\endaligned$$ 
\end{cor}
\begin{proof} The first claim was shown within the proof of Theorem~\ref{existDaFa}. That  the limit defining the directional derivative $\partial^a_e \phi_a = \lim_{t\to 0}\eta^e_t$ can be taken in $W^{1, 2}_\loc(\C)$-metrics, too, requires a further short argument.

Let us denote $L_a = \partial_{\zbar}-\mu_a(z)\,\partial_z -\nu_a(z)\,\overline{\partial_z }$. We saw in the proof of Theorem~\ref{existDaFa} that
$$\aligned
L_a (\eta^e_{t_j}(z))&= h_{t_j}(z)\\
L_a(\eta^e)&=0,
\endaligned$$
almost everywhere.  We know also from \eqref{extra16}
that for every $R>0$,
\begin{equation}\label{strongL2}
\lim_{j\to\infty}\|h_{t_j}\|_{L^s(\D(0,R))}=0, \qquad \forall s < \frac{2K}{K-1}.
\end{equation}

Now, choose a real valued cutoff function $\xi\in C^{\infty}_0(\D(0,2R))$, $\xi=1$ on $\D(0,R)$. Then
$$L_a(\xi\,(\eta^{e}_{t_j}-\eta^e)) = (\eta^e_{t_j}-\eta^e) \, L_a(\xi) + \xi\, h_{t_j}.$$
We can write $L_a = \mathcal{B}\partial_{\zbar} =\partial_{\zbar}-\mu_a(z)\,\cS\partial_{\zbar} -\nu_a(z)\,\overline{\cS \partial_{\zbar}}$ where $|\mu_a(z)| + |\nu_a(z)| \leq k < 1$ and the Beurling transform $\cS$ is an isometry in $L^2$. Hence  $\mathcal{B}$ is invertible in $L^2(\C)$, so that for every  $\xi\,(\eta^{e}_{t_j}-\eta^e)\in W^{1,2}(\C)$, we have an a priori $L^2$-estimate
$$
\|\partial_{\zbar}(\xi(\eta^{e}_{t_j}-\eta^e))\|_{L^2(\C)}\leq c(K)\,\left(\|(\eta^e_{t_j}-\eta^e) \,  L_a(\xi) \|_{L^2(\C)}+ \|\xi\, h_{t_j}\|_{L^2(\C)}\right).
$$
The first term on the right hand side converges to $0$, due to the local uniform convergence of $\eta^e_{t_j}\to \eta^e$. The second term also converges to $0$ by  \eqref{strongL2}. Using now that the Beurling transform $\cS$ is bounded in $L^2(\C)$, we get for the full differentials that
\begin{equation*}
\lim_{j\to\infty}\left\|D_z \eta^e_{t_j} - D_z \eta^e\right\|_{L^2(\D(0,R))}=0.
\end{equation*}
\end{proof}

\begin{rem}
The argument above actually shows that Sobolev space $W^{1,2}_{\loc}(\C)$ can be replaced by $W^{1,p}_{\loc}(\C)$ for any $2<p<\frac{2K}{K-1}$, if one uses deeper invertibility properties of the operator $\mathcal{B}$, see \cite[Section 14]{AIM}. 
\end{rem}

It turns out that the derivative  is continuous as a function of $a$ respect to various topologies. 

\begin{thm}\label{partialder}
Assume that $\cH$ has the uniqueness property. Let $w \mapsto \cH(z, w)$ be $C^1$ for every fixed $z\in \C$. Then
\begin{enumerate}
\item[(a)] $a \mapsto \phi_a$ is continuously (Fr\'echet) differentiable as  a map  $\C \to L^{\infty}_\loc(\C)$;
\item[(b)] $a\mapsto\phi_a(z)$ is continuously differentiable on $\C$, for every $z\in\C$;
\item[(c)] $a \mapsto \phi_a$ is continuously (Fr\'echet) differentiable as  a map  $\C \to W^{1,2}_{\loc}(\C)$.
\end{enumerate}
\end{thm}
\begin{proof}
$(b)$ follows from $(a)$. For $(a)$ and $(c)$ we know by Corollary~\ref{FrechetW12} and Remark \ref{extra17} that $a\mapsto \phi_a$ is Gateaux differentiable. Arguing as in the proof of Corollary \ref{corcontdif}, the bi-Lipschitz properties of $a \mapsto \phi_a$ 
with  \cite[Proposition 4.3]{BL} show that the map $a \mapsto \phi_a$ is Fr\'echet differentiable from $\C$ to both of the spaces $L^\infty_{\loc}(\C)$ and $W^{1,2}_{\loc}(\C)$.

It remains to show the continuity of the derivatives, that 
\begin{equation}\label{first}
\lim_{b \to a}\| \partial_e^a\phi_a - \partial_e^a\phi_b \|_X = 0
\end{equation}
whether $X$ is  $L^\infty(\D(0, R))$ or $W^{1, 2}(\D(0, R))$.
Note that
\eqref{first} yields the continuity of $a \mapsto D_a\phi_a : \C \to L(\C, X)$. 

Here the case of $X = L^\infty(\D(0, R))$ follows directly from Remark \ref{extra17} and  Proposition~\ref{bilip} $(b)$. 
For  $X = W^{1, 2}(\D(0, R))$ we need to use the Beltrami equations \eqref{rlineara}. 
We first claim that if $b_j \to a$, then 
for every $1<p<\infty$,
\begin{equation}\label{conv1}\aligned
&\|\mu_{b_j} - \mu_a \|_{L^p(\D(0, R))}\\
&\quad= \|\partial_w \cH(\cdot , \partial_z \phi_{b_j}) - \partial_w \cH(\cdot, \partial_z \phi_a )\|_{L^p(\D(0, R))} \to 0,\\
&\|\nu_{b_j} - \nu_a \|_{L^p(\D(0, R))}\\
&\quad=  \|\partial_{\wbar} \cH(\cdot,\partial_z \phi_{b_j} ) - \partial_{\wbar} \cH(\cdot,\partial_z \phi_a) \|_{L^p(\D(0, R))} \to 0,
\endaligned
\end{equation}
as $j\to\infty$.  In fact, by  Proposition~\ref{bilip} $(c)$, we know that $\partial_z \phi_{b_j}(z) \to \partial_z \phi_a(z)$ for almost every $z\in \D(0, R)$, at least for a subsequence. 
Since $w\mapsto D_w\cH(z,w)$ is continuous and bounded by assumption, the dominated convergence theorem gives \eqref{conv1}  for this subsquence. But since the limit of $\mu_{b_j}$\!'s is the same for every converging subsequence, \eqref{conv1} holds for the entire sequence $(b_j)$.

On the other hand, 
let us consider the operator $L_a = \partial_{\zbar}-\mu_a(z)\,\partial_z -\nu_a(z)\,\overline{\partial_z }$ as in the proof of Corollary \ref{FrechetW12}. If  $\xi\in C^{\infty}_0(\D(0,2R))$ is  a real valued cutoff function with $\xi=1$ on $\D(0,R)$, this time we obtain
$$L_a(\xi\,(\partial_e^a\phi_{b_j}-\partial_e^a\phi_a)) = (\partial_e^a\phi_{b_j}-\partial_e^a\phi_a) \, L_a(\xi) + \xi\, L_a(\partial_e^a\phi_{b_j}).$$
Here we already saw that the first term on the right converges to zero uniformly. For the second term on the right,
$$
L_a(\partial_e^a\phi_{b_j}) = (\mu_{b_j}(z)-\mu_{a}(z))\partial_z \partial^{a}_{e}\phi_{b_j}(z) +(\nu_{b_j}(z)-\nu_{a}(z))\overline{\partial_z \partial^{a}_{e}\phi_{b_j}(z)}.
$$
The $\partial_e^a\phi_{b_j}$\!'s are $K$-quasiconformal maps of $\C$ fixing $0$ and taking $1$ to $e$, thus \cite[Corollary 13.2.4]{AIM} gives 
   for every $b \in \C$
\begin{equation*}
\|D_z \partial_e^a \phi_b \|_{L^p(\D(0,R))} \leq c_p(e,K, R), \qquad  p < \frac{2K}{K-1}.
\end{equation*}
Combining this with \eqref{conv1} shows that $L_a(\partial_e^a\phi_{b_j}) \to 0$ in $L^2(\D(0,R))$, and in fact in $L^p(\D(0,R))$ for every $ p < \frac{2K}{K-1}$.

Now the same argument as in Corollary \ref{FrechetW12} gives us a priori $L^2$-estimates
$\|\partial_{\zbar}(\xi\,(\partial_e^a\phi_{b_j}-\partial_e^a\phi_a))\|_{L^2(\C)} \to 0$ as $b_j \to a$, and similarly as there we obtain 
\begin{equation}\label{frechetlimit4}
\lim_{j\to\infty}\left\|D_z \partial_e^a\phi_{b_j} - D_z\partial_e^a\phi_{a} \right\|_{L^2(\D(0,R))}=0.
\end{equation}
\end{proof}

It is interesting to compare Theorem~\ref{partialder} with Corollary~\ref{corcontdif}. The fact that we are starting 
with the (nonlinear) equation allows us to get the continuity of the derivative in $W^{1,2}_{\loc}(\C)$. In the manifolds language this enables us to embed $\cF_{\cH}$ in $W^{1,2}_{\loc}(\C)$. 

\begin{cor}
Assume that $\cH$ has the uniqueness property and that $w \mapsto \cH(z, w) \in C^1(\C)$ for every fixed $z\in \C$. Then $\cF_{\cH}$ is a $C^1$-embedded submanifold of $W^{1,2}_{\loc}(\C)$.
\end{cor}
\begin{proof}
By Theorem~\ref{partialder}, $a \mapsto \phi_a \in C^1(\C, W^{1,2}_{\loc}(\C))$.
Proving that that this yields an embedded manifold is similar to Proposition~\ref{mani}. Indeed, the proof that  $a\mapsto\phi_a$ is an immersion is the same and  the fact that $a \mapsto \phi_a : \C \to W^{1, 2}_{\loc}(\C)$ is a topological  embedding follows by Proposition~\ref{bilip} $(c)$.
\end{proof}

\section{Smooth $\cH$-equations}\label{smoothHequations}

\noindent We study the smoothness of  the  field $\cF_{\cH} = \{\phi_a(z)\}_{a\in\C}$ associated to the nonlinear Beltrami equation \eqref{Hqr}, when we are given a H\"older smooth   structure function $\cH$. The main results are Corollary~\ref{c3} and Theorem~\ref{c2}, in which we derive the smoothness of $\phi_a(z)$ with respect to both variables $a$ and $z$.

We will repeatedly use the following interpolation estimates.

\begin{thm}[Interpolation]\label{interpolation}
Let $\D_r \subset \C$ be a disk, and $f \in L^p(\D_{2r}) \cap C^{\gamma}(\D_{2r})$ for some $1<p<\infty$, $0<\gamma<1$. Given any $\theta\in\left(\frac{2}{2+\gamma p}, 1\right]$, set
\begin{equation*}
s = \theta\left(\gamma+\frac2p\right)-\frac2p.
\end{equation*}
Then
$$
\|f\|_{L^\infty(\D_{r})} \leq \|f\|_{C^s(\D_r)}\leq c(\theta, \D_{2r})\,\|f\|_{L^p(\D_{2r})}^{1-\theta}\,\|f\|_{C^\gamma(\D_{2r})}^\theta.$$
\end{thm}

\begin{proof}
Notice that $C^{\gamma} = B^{\gamma}_{\infty\infty}$, $L^p \subset B^0_{p, \infty}$, and, by the embedding theorem \cite[Theorem 6.5.1]{interpolation}, there is a continuous inclusion
$B^{\theta\gamma}_{\frac{p}{1-\theta}, \infty} \subset B^{s}_{\infty, \infty}$, where $B^{\cdot}_{\cdot, \cdot}$ denotes the Besov space.
In particular, if $\frac{2}{2+\gamma p}<\theta<1$, then $B^s_{\infty,\infty}=C^s$ with  $0<s<1$. By \cite[Theorem 6.4.5]{interpolation},
$$
(B^0_{p, \infty}, B^{\gamma}_{\infty, \infty})_{[\theta]} = B^{\theta\gamma}_{\frac{p}{1-\theta}, \infty}, \qquad 0\leq \theta\leq 1.
$$ 
Thus, for $f \in L^p(\C) \cap C^{\gamma}(\C)$,
$$
\|f\|_{L^\infty(\C)} \leq \|f\|_{C^s(\C)}\leq c(\theta)\,\|f\|_{L^p(\C)}^{1-\theta}\,\|f\|_{C^\gamma(\C)}^\theta.
$$
The statement follows when we apply the previous estimate to $\xi f$ with a suitable cutoff function  $\xi$.
\end{proof}

We first combine  the  interpolation (Theorem~\ref{interpolation}) with the Schauder estimates (Theorem~\ref{schauder}) to obtain continuity of $a \mapsto D_z\phi_a : \C \to L^{\infty}_{\loc}(\C)$; let the constant $\gamma(\alpha,K)$ be as in Theorem \ref{schauder}.

\begin{lem}\label{cont}
Assume that the structure function  $\cH$ has the uniqueness property. Let $\cH$ satisfy \eqref{holdercondition} and $\cF_\cH =\{\phi_a\}_{a\in\C}$. 
Then, for every $0<s<\gamma(\alpha, K) \leq \alpha$,
$$
\|D_z\phi_{a} - D_z\phi_{b}\|_{C^s(\D(z_1,R))} \leq c(\theta, \cH, z_1, R)\,|a - b|^{1 - \theta} (|a| + |b|)^\theta,
$$
where $\theta = \theta(s, \gamma)$.  

\end{lem}

\begin{proof}
By Proposition~\ref{bilip} $(c)$,
$$
\|D_z\phi_{a} - D_z\phi_{b}\|_{L^p(\D(z_1, R))} \leq c(K, \D(z_1, R))\,|a - b|.
$$
Further, we have a  $C^{1, \gamma}(\D(z_1, R))$-bound for $\phi_a$ and $\phi_{b}$ that is uniform in $|a|$, $|b|$, by combining the Schauder norm estimate \eqref{thmnorm} and Proposition~\ref{bilip} $(c)$,
$$
\|D_z\phi_{a}\|_{C^\gamma(\D(z_1,R))}
\leq c(\cH, \D(z_1,2R))\,\|D_z\phi_{a}\|_{L^2(\D(z_1,2R))} \leq c(\cH, \D(z_1,2R))\,|a|.
$$
Hence we can use interpolation (Theorem~\ref{interpolation}) to get the claim. 
\end{proof}

\begin{cor}\label{c3}
Assume that the structure function  $\cH$ has the uniqueness property. Let $\cH$ satisfy \eqref{holdercondition}   and $\cF_\cH=\{\phi_a(z)\}_{a\in\C}$. Then
$$
(z, a) \mapsto \phi_a(z) \in C^1(\C \times \C).
$$
\end{cor}

\begin{proof}
That $(z, a) \mapsto D_z\phi_a(z)$ exists and is continuous on $\C \times \C$ follows from Lemma~\ref{cont}. Concerning  $(z, a) \mapsto D_a\phi_a(z)$, the existence and continuity in $a$ is given by Theorem \ref{partialder} while continuity in $z$ follows from  the fact that all  partial derivatives $\partial_e^a \phi_a(z)$ are quasiconformal as maps in the $z$-variable, see Theorem~\ref{partialder1}.
 \end{proof}

Let us then see how the smoothness of $\cF_{\cH}$ is improved  when $\cH$ is a regular  structure function (see Definition~\ref{defregular}). The regularity of $\cH$ is required to guarantee that the Beltrami coefficients \eqref{muanua} are H\"older continuous, which enables us to use the classical Schauder estimates. 

\begin{lem}  \label{c1}
Let $\cH$ be  regular   and $\cF_\cH=\{\phi_a(z)\}_{a\in\C}$. Then
$$z\mapsto D_z D_a\phi_a(z) \quad \text{is locally H\"older continuous}$$
and, moreover, 
\begin{equation}\label{uniformbound}
\|D_z D_a\phi_a\|_{C^{\gamma\alpha}(\D_r)} \leq c(\cH, |a|, \D_r),
\end{equation}
where $\gamma 
\leq \alpha$.
\end{lem}

\begin{proof}
First note that since $w \mapsto \cH(z, w) \in C^1(\C)$, we use  \cite[Theorem 1.3]{ACFKJ} to apply the Schauder estimates of Theorem~\ref{schauder} with $\gamma(\alpha, K) = \alpha$.


We know by Theorem~\ref{partialder1} 
that for $e = 1$ and $e= i$ 
the directional derivative $f=\partial^a_e\phi_a$ is a $K$-quasiconformal solution to the $\R$-linear  Beltrami equation  
\begin{equation*}
\partial_{\zbar}f(z) = \mu_a(z)\,\partial_z f(z) + \nu_a(z)\,\overline{\partial_z f(z)} \qquad \text{a.e.},
\end{equation*}
where $\mu_a$ and $\nu_a$ are given in \eqref{muanua}. 

By the Schauder estimates of Theorem~\ref{schauder} and  Proposition~\ref{bilip} $(c)$, there exists a constant $c = c(\cH, \D(z_0, r))$ such that 
\begin{equation*}
\|D_z \phi_a \|_{C^\gamma(\D(z_0, r))} \leq c\,|a| \eta_{K}( r). 
\end{equation*}
Further, using the assumption that $\cH$ is regular, so that  $(z, w) \mapsto D_w\cH(z, w)$ is locally $\alpha$-H\"older continuous,  one gets for  $z_1,z_2\in \D(z_0, r) = \D_r$ 
\begin{equation}\label{muaholder}
\aligned
&|\mu_{a}(z_1)-\mu_a(z_2)|
= |\partial_w\cH(z_1,\partial_z\phi_a(z_1))-\partial_w\cH(z_2,\partial_z\phi_a(z_2))|\\
&\qquad\leq c(\D_r, \D_{\| D_z \phi_a \|_{L^\infty(\D_r)}})\,\left(|z_1-z_2|^{\alpha} + |\partial_z\phi_a(z_1)-\partial_z\phi_a(z_2)|^\alpha\right)\\
&\qquad\leq c(\cH, |a|, \D_r)\,\left(|z_1-z_2|^{\alpha} + [D_z\phi_a]^\alpha_{C^\gamma(\D_r)}|z_1- z_2|^{\gamma\alpha}\right) \\
&\qquad\leq c(\cH, |a|, \D_r)\,|z_1 - z_2|^{\gamma\alpha}.
\endaligned
\end{equation}
Hence $\mu_a\in C^{\gamma\alpha}_{\loc}(\C)$   with uniform norm bounds 
Therefore, by the classical Schauder estimates for linear Beltrami equations \cite[Theorem 15.0.6]{AIM}, $D_z f \in C^{\gamma\alpha}_{\loc}(\C)$ with  estimate $\|D_z f\|_{C^{\gamma\alpha}(\D_r)} \leq c(\cH, |a|, \D_r)\|f\|_{L^2(\D_{2r})}$.  Bounding the $L^2$-norm with Proposition~\ref{bilip} $(c)$ proves then our statement. 
\end{proof}


\begin{thm}\label{c2}
Let $\cH$ be  regular   and $\cF_\cH=\{\phi_a(z)\}_{a\in\C}$. Then
$$(z,a) \mapsto D_a D_z\phi_a(z) = D_z D_a \phi_a(z) \quad \text{is continuous on $\C \times \C$.}$$
\end{thm}

\begin{proof}
By Corollary \ref{c3} the partial derivatives  $D_a\phi_a(z)$ and $D_z\phi_a(z)$ exist and are continuous in $\C \times \C$.

To show the continuity of $(z, a) \mapsto D_z D_a \phi_a(z)$ we need
to pass from boundedness shown in Lemma~\ref{c1} to continuity. We combine  interpolation (Theorem~\ref{interpolation}), continuity in  $W^{1,2}_{\loc}(\C)$-topology (Theorem~\ref{partialder} $(c)$)
and \eqref{uniformbound}.  Indeed,  let  us fix a direction $e\in\C$, $|e|=1$, and $a, b\in \D_R$, and denote $f_b=\partial^a_e\phi_{b}$ and $f_a=\partial^a_e\phi_a$. We get 
$$\aligned
&\|D_zf_b-D_zf_a\|_{C^s(\D_r)}\\
&\quad\leq c(\theta, \D_r)\,\|D_zf_b-D_zf_a\|^{1-\theta}_{L^2(\D_{2r})}\,\|D_zf_b-D_zf_a\|_{C^{\gamma\alpha}(\D_{2r})}^\theta\\
&\quad \leq c(\theta, \cH,\D_r, \D_R)\,\|D_zf_b-D_zf_a\|^{1-\theta}_{L^2(\D_{2r})}\endaligned$$
and the continuity in $\D_r \times \D_R$ follows from \eqref{frechetlimit4}.

\smallskip

Finally, since $(z, a) \mapsto D_z D_a \phi_a(z)$ is continuous on $\C \times \C$, we can change the order of differentiation, i.e., $D_a D_z\phi_a(z) = D_z D_a \phi_a(z)$, see, for example, \cite[pp. 235--236]{rud}.
\end{proof}

\section{From $\cF$ to $\cH$}\label{famsec}

\noindent The main goal of this section is  to prove Theorems~\ref{Uniquenessthm} and \ref{homeo} for  fields $\cF_{\cH}$ associated to the structure function $\cH$. It turns out that to obtain a unique and well-defined structure function $\cH_{\cF}$ from a field $\cF$, one first needs some  smoothness properties, and it is natural to assume\begin{equation}\label{smooth}
(z,a) \mapsto D_a D_z\phi_a(z)  \quad \text{is continuous on $\C \times \C$,} \quad \phi_a \in \cF,
\end{equation}
since this condition is satisfied by fields arising from a regular $\cH$. But more importantly, we need  $\cF$ to  satisfy conditions of  non-degeneracy.

\begin{defn}\label{nondegdef}
Suppose that $\cF=\{\phi_a(z)\}_{a\in\C}$ is a field of $K$-quasiconformal maps satisfying (F1), (F2) from Definition~\ref{perhe}. We say that $\cF$ is \textit{smooth} if in addition
\eqref{smooth} holds. On the other hand, we call  $\cF$ \textit{non-degenerate}, 
if  the following two additional conditions are satisfied.
\begin{enumerate}
\item For every disk $\D(0, R)$ there exists a constant $c = c(K, R)$ such that
$$\frac{1}{c}\leq \frac{|\partial_z\phi_a(z)|}{|a|}\leq c,\qquad z\in \D(0, R),\quad a\neq 0.$$
\item For every $z,a\in\C$,  one has
$$\det D_a(\partial_z\phi_a)(z) \neq 0.$$
\end{enumerate}
\end{defn}

For non-degenerate and smooth fields we can use topology to understand the range of $a \mapsto \partial_z\phi_a(z)$.
\begin{lem}\label{keylemma}
Let $\cF=\{\phi_a(z)\}_{a\in\C}$ be a smooth and non-degenerate field of $K$-quasi\-conformal maps. 
Then for every $z \in \C$,  the mapping
$$a\mapsto \partial_z\phi_a(z)=: F_z(a)$$
is a homeomorphism of $\C$. 
\end{lem}
\begin{proof} By our assumptions  $F_z: a\mapsto \partial_z\phi_a(z)$ is well-defined and continuously differentiable, for every fixed $z\in\C$, so that  the above  non-degeneracy assumptions make it   a locally homeomorphic  proper map  of $\C$. Thus by the monodromy theorem $F_z$ is a global homeomorpism.
\end{proof}

In particular, we want to apply the above argument to a field of solutions to a Beltrami equation.
 This is done by the following result, whose proof is postponed 
 to the next Subsection~\ref{nondegfamilies}. 

\begin{prop}\label{FHisnondeg}
Let $\cH$ be a regular  structure function. Then the field $\cF_\cH$  associated to $\cH$ is non-degenerate.
\end{prop}

\begin{proof}[Proof of Theorem~\ref{homeo}]
It simply suffices to notice that for  a regular  structure function $\cH$, the field  $\cF_\cH = \{\phi_a(z) \}$ satisfies all the requirements in Lemma~\ref{keylemma}, due to 
Theorem~\ref{c2} and Proposition~\ref{FHisnondeg}.
\end{proof}

We know that the   structure functions $\cH$ with the uniqueness property generate a unique field $\cF_{\cH}$. For the other direction,  fields $\cF$ that are smooth enough and non-degenerate determine $\cH_\cF$ uniquely. 

\begin{thm}\label{family}
Assume that $\cF = \{\phi_a\}_{a\in\,\C}$ is a  field of $K$-quasiconformal mappings which is both smooth and non-degenerate in the sense of Definition \ref{nondegdef}.

Then there is a unique $\cH = \cH_{\cF}$ satisfying (H1), (H2) such that every member of $\cF$ is a $K$-quasiconformal solution to
\begin{equation*}
\partial_{\zbar} \phi_a(z) = \cH(z, \partial_z \phi_a(z)).
\end{equation*}
\end{thm}

We will prove the above theorem in Section~\ref{famsecsmooth}. With it we have then completed the argument  that starting from a regular   structure function $\cH$, the associated field $\cF_\cH$ defines $\cH$ uniquely.

\begin{proof}[Proof of Theorem~\ref{Uniquenessthm}]
Assuming $\cH$ to be a regular  structure function, we obtain from Proposition~\ref{FHisnondeg} that the  field $\cF_\cH = \{\phi_a(z)\}_{a \in \C}$ associated to $\cH$ is non-degenerate. Further,  the required smoothness properties are provided by  Theorem~\ref{c2}, so that $\cF_{\cH}$ defines $\cH$ uniquely by Theorem~\ref{family}.
\end{proof}

\begin{rem}\label{uniq2} Suppose $\cF$ is a  field of $K$-quasiconformal mappings, with  distortions  small enough, e.g. $K < \sqrt{2}$, or more generally, such that  \eqref{inftybound} holds for $\cH_{\cF}$.
Then the structure function $\cH_{\cF}$ has  the uniqueness property by  \cite[Theorem 1.2]{ACFJS}. If in addition $\cF$ is smooth and non-degenerate in the sense of  Definition \ref{nondegdef},  then the corresponding field  $\cF_{{\cH}_{\cF}}$ must be equal to $\cF$.
\end{rem}


\subsection{Field $\cF_{\cH}$ is non-degenerate}\label{nondegfamilies}

We will prove Proposition~\ref{FHisnondeg} in two steps (Corollary~\ref{continfty} and Proposition~\ref{locinj2}). 
The first is   a consequence of Theorems~\ref{schauder} and \ref{Jac}.

\begin{cor}\label{continfty}
Let  $\cH$  satisfy \eqref{holdercondition} and $\cF_\cH = \{ \phi_a(z) \}$. There exists a constant $c=c(K,R)>0$ such that
\begin{equation}\label{lineargrowth}
\frac1{c}\leq \frac{|\partial_z\phi_a(z)|}{|a|}\leq c, \qquad z\in \D(0, R).
\end{equation}
In particular, $a \mapsto \partial_z \phi_a(z)$ admits a continuous extension at the point $a = \infty$.\end{cor}
\begin{proof}

The mappings $f=\frac1a\,\phi_a$ are normalized quasiconformal solutions ($0 \mapsto 0$ and $1 \mapsto 1$) to the following nonlinear Beltrami equations
$$
\partial_{\zbar} f(z) = \tilde{\cH}(z,\partial_z f(z)) \qquad \text{a.e.},
$$
where $\tilde{\cH}(z,w)=\frac1a\,\cH(z,aw)$. Clearly $\tilde{\cH}$ satisfies (H1) and (H2). Moreover, if $z_1, z_2\in\Omega$, the H\"older bounds for $\cH$ give us that
$$
\left|\tilde{\cH}(z_1,w) -\tilde{\cH}(z_2, w)\right| \leq\frac{\mathbf{H}_{\alpha}(\Omega)}{|a|}\,|z_1 - z_2|^\alpha\,|a\,w| = \mathbf{H}_{\alpha}(\Omega)\,|z_1 - z_2|^\alpha|w|.
$$
In particular the H\"older constant of $\tilde{\cH}$ does not depend on $a$. Thus, first by Theorem~\ref{schauder} and Proposition~\ref{bilip} $(c)$ (with $b=0$) we get an upper bound
$$
\|D_z f \|_{C^{\gamma}(\D(0, R))} \leq c\;\|D_z f\|_{L^{2}(\D(0, 2R))} \leq c,
$$
where $c = c(\cH, R)$.
Since  $\alpha$, $\gamma$ and the   structure function are fixed and $K = \frac{1 + k}{1 - k}$, the upper bound at \eqref{lineargrowth} follows. For the lower bound, we use Theorem~\ref{Jac} and get
$$
(1 - k)\left|\frac{\partial_z\phi_a}{a}\right| \geq J\left(z, \frac{\phi_a}{a}\right) \geq c(\cH, R) > 0,
$$
for every $z\in\D(0, R)$, which proves \eqref{lineargrowth}. 

Finally, by Lemma~\ref{cont}, $a \mapsto \partial_z \phi_a(z)$ is continuous in $\C$, uniformly on compact subsets of $z$.  The estimate \eqref{lineargrowth} yields that
  $\lim_{|a| \to \infty} |\partial_z \phi_a(z)| = \infty$ for every $z\in\C$. 
\end{proof}

We are left to show the non-vanishing of the Jacobian.
Note that the mapping $a\mapsto \partial_z\phi_a(z)$, even being continuously differentiable, might still have a vanishing Jacobian for some $a$ and $z$ as, for instance, $\phi_a(z)=a\,z|z|^2$. However, the fact that the partial derivatives are solutions to a linear Beltrami equation with H\"older continuous coefficients prevents this kind of behaviour.
We find the following surprising   application of the non-vanishing of new null Lagrangians very interesting. These null Lagrangians have been  a recent theme of research (see, e.g., \cite{G-cl}, \cite{AIM}, \cite{AN}, \cite{AJ}, \cite{J}).

\begin{prop}[Proposition~\ref{locinj}]\label{locinj2}
Let $\cH$ be a regular  structure function and $\cF_\cH = \{ \phi_a(z) \}$. Then, for $a$, $z\in\C$,
$$\big|\det[D_a \partial_z \phi_a(z) ]\big|  \geq c(\cH, |a|, \D_r) > 0, $$ and the determinant does not change sign.
\end{prop}

\begin{proof} 
Let 
\begin{align*}
f(z) &:= \partial_{1}^{a}\phi_a(z) =  \partial_a \phi_a(z) + \partial_{\bar{a}} \phi_a(z)\\
g(z) &:= -\partial_{i}^{a}\phi_a(z) =  i\big(\partial_{\bar{a}} \phi_a(z) - \partial_a \phi_a(z)\big).
\end{align*}
Then 
$$
\Im \left(\partial_z f\,\overline{\partial_z g}\right)= |\partial_a \partial_z\phi_a|^2 - |\partial_{\bar{a}}\partial_z\phi_a|^2 = \det[D_a \partial_z \phi_a(z) ], 
$$
since we can exchange the order of differentiation by Theorem~\ref{c2}.

Now, $f$ and $g$ solve the  $\R$-linear Beltrami equation \eqref{rlinearaint} by  Theorem~\ref{partialder1}, and from  \eqref{muaholder} we know that $\| \mu_a \|_{\C^{\gamma\alpha}(\D_r)} + \| \nu_a \|_{\C^{\gamma\alpha}(\D_r)} \leq c(\cH, |a|, \D_r)$. As a consequence, by \cite[Proposition~5.1]{G-cl}, we obtain
$$
\Im(\partial_z f(z)\,\overline{\partial_z g(z)})\neq 0 \qquad \text{everywhere}
$$
and it does not change sign, see \cite[Lemma 7.1.]{G-cl} or \cite[Theorem 6.3.2]{AIM}.

Actually, the argument  from  \cite[Lemma 7.1.]{G-cl}  can easily be made quantitative. Let $(t, s) \in S^1$, then $|\Im(\partial_z f(z)\,\overline{\partial_z g(z)})|^2$ is the discriminant of the quadratic form $|t\,\partial_z f(z) + s\,\partial_zg(z)|^2$. We have for the first eigenvalue $\lambda_1$ that
$$
\iota := \inf_{(t, s)\in S^1} \frac{|t\,\partial_z f(z) + s\,\partial_zg(z)|^2}{t^2 + s^2} = \lambda_1
$$
and thus
$$
|\Im(\partial_z f(z)\,\overline{\partial_z g(z)})|^2 = 4\left((|\partial_z f(z)|^2 + |\partial_z g(z)|^2)\iota - \iota^2\right).
$$
As in  \cite[Proposition~5.1]{G-cl}, map $(t\, f + s\,g)^{-1}$ solves also an $\R$-linear Beltrami equation  with H\"older continuous coefficients for every $t, s$. Moreover, it
is easy to check that, in our case, the H\"older norm of the coefficients is bounded by $c(\cH, |a|, \D_r)$.

Now, from our lower bound for the Jacobian (Theorem \ref{Jac}) or the classical Schauder estimates, one bounds $\iota$ from below by the H\"older norm of the coefficients and the $L^2$-norm of the solution and hence
$$
|\Im(\partial_z f(z)\,\overline{\partial_z g(z)})| \geq c(\cH, |a|, \D_r).
$$
\end{proof}

\begin{rem} If we strengthen our assumptions by a condition of the differential quotients like\begin{equation*}
\aligned&\left|
\frac{\cH(z_1,w_1+h_1)-\cH(z_1,w_1)}{h_1}-
\frac{\cH(z_2,w_2+h_2)-\cH(z_2,w_2)}{h_2}
\right|\\
&\qquad\leq c(\D_r)\,\bigg(|z_1-z_2|^\alpha+|w_1-w_2|^\beta\bigg)
\endaligned
\end{equation*} 
for every $z_1,z_2\in \D_r \subset \C$ and $w_1,w_2\in\C$, $h_1,h_2\in\C\setminus\{0\}$,  then it actually  holds that
$$\frac{1}{c_a(K,\D_r)}\leq \left|\frac{\partial_z\phi_a(z)-\partial_z\phi_b(z)}{a-b}\right|\leq c_a(K,\D_r),\qquad z\in \D_r,\, a\neq b.$$
and thus $a\mapsto\partial_z\phi_a(z)$ is directly a local homeomorphism in $\hat{\C}$ and thus a global homeomorphism.
\end{rem}

\subsection{Construction of $\cH_{\cF}$}\label{famsecsmooth}
Since it will be repeatedly used in this section, we need a label for the inverse of the homeomorphism $\, a \mapsto F_z(a) := \partial_z\phi_a(z)$, see Lemma \ref{keylemma}.

\begin{defn}\label{adef}
Let $\cF=\{\phi_a(z)\}_{a\in\C}$ be a smooth and non-degenerate field of $K$-quasiconformal maps. Then 
for each $(z,w) \in \C \times \C$ we denote 
\[a(z,w)=F^{-1}_z(w). \]
\end{defn}
\smallskip

Under the assumptions of Theorem~\ref{family}, the existence, uniqueness, and regularity properties of $\cH_{\cF}$ depend on those of the function 
$a$. Indeed, we need to have  
$$\cH(z,\partial_{z}\phi_a(z))= \partial_{\zbar}\phi_a(z),$$
and since $a \mapsto F_z(a) = \partial_{z}\phi_a(z)$ is a homeomorphism of $\C$, necessarily
$\cH = \cH_{\cF}$ is determined by 
\begin{equation}\label{defH}
  \cH(z, w)=\partial_{\zbar} \phi_a(z)\vert_{a=a(z,w)}.
\end{equation}
Equivalently, $(z,w)\mapsto \cH(z,w)$ is defined as the composition of 
\begin{equation}\label{compH}
(z,w)\mapsto a(z,w)\qquad \text{with} \qquad a\mapsto\partial_{\zbar}\phi_a(z).
\end{equation}


Thus $\cH(z, w)$ is uniquely defined by $\cF$ in $\C \times \C$, and it only remains to show that $\cH$ has all the properties required by Theorem \ref{family}. For this 
we  first prove the continuity of $a(z, w)$ in both variables.

\begin{lem}\label{continuityofa}
Let $\cF = \{ \phi_a(z) \}_{a\in\C}$ be a non-degenerate field of quasiconformal maps that  satisfies the smoothness condition \eqref{smooth}. Then
\begin{enumerate}
\item[(a)] $w\mapsto a(z,w)$ is continuous and
\item[(b)] $z\mapsto a(z,w)$ is continuous.
\end{enumerate}
\end{lem}
\begin{proof} The first claim, $(a)$, is just a restatement of Lemma~\ref{keylemma}.
For  $(b)$ fix $z_0, w\in\C$. Take a sequence $z_n \to z_0$, $z_n \in \D(z_0, r)$. By Lemma~\ref{keylemma}, we have for  $a_0:= a(z_0,w)$ and $a_n := a(z_n,w)$  that $F_{z_0}(a_0) = \partial_{z}\phi_{a_0}(z_0) = w$ and $F_{z_n}(a_n) = \partial_{z}\phi_{a_n}(z_n) = w$. Now,
\begin{equation}\label{acont}
|a_0 - a_n| = |F_{z_n}^{-1} \circ F_{z_n} \circ F_{z_0}^{-1}(w) - F_{z_n}^{-1}(w)| =: |F_{z_n}^{-1}(w_n) - F_{z_n}^{-1}(w)|.
\end{equation}
Here
\begin{align}\label{holderforFz}
|w_n - w| &= |F_{z_n}(a_0) - F_{z_0}(a_0)| = |\partial_{z}\phi_{a_0}(z_n) - \partial_{z}\phi_{a_0}(z_0)|\to 0, \end{align}
when  $n\to\infty$, because  for fixed $a_0\in \C$, $z \mapsto \phi_{a_0}(z)\in C^1(\C)$. Furthermore, the smoothness
condition \eqref{smooth} implies $F_{z_n}$ is locally uniformly Lipschitz in a neighbourhood of $(z_0,w)$. Thus, if we can show that the Jacobian of $F_{z_n}$ with respect to $a$ has a positive lower bound, then $F_{z_n}^{-1}$
is  locally uniformly Lipschitz, and the continuity follows from \eqref{acont} and \eqref{holderforFz}.



\smallskip

But the estimates for the Jacobian of $F_{z_n}$ quickly follow. First, by  Condition (1) in the  non-degeneracy  Definition~\ref{nondegdef}, 
\begin{equation}\label{aboundbyw}
|a_n| \leq c(K, |z_0|, r)\,|F_{z_n}(a_n)| =  c(K, |z_0|, r)\,|w|.
\end{equation}
Hence we can study the situation locally for  $a_n \in \overline{\D\left(0, c(K, |z_0|, r)\,|w|\right)}$.

The Jacobian $J(a, F_{z_0}) \neq 0$ due to Condition (2) in the definition of non-degeneracy, and the smoothness \eqref{smooth} shows that $(z,a) \mapsto J(a, F_{z})$ is continuous. We may assume without loss of generality $J(a, F_{z_0}) > 0$, and then  
by compactness and continuity,  
$$
J(a, F_{z_0}) \geq c_2 = c_2(|w|, c(K, |z_0|, r)), \qquad a \in \overline{\D\left(0, c(K, |z_0|, r)\,|w|\right)}.
$$
Hence, when $r$ is small enough and $|z_n - z_0| < r$, we obtain 
$J(a, F_{z_n}) \geq  c_2/2  > 0$. We have shown our claim.
\end{proof}

\begin{proof}[Proof of Theorem~\ref{family}]
By assumption $z\mapsto \phi_a(z)$ is $C^1$, so that $\partial_{\zbar}\phi_a(z)$ exists at every point $z$ and $a$, and hence 
the function $w\mapsto \cH(z,w)$ given by \eqref{defH}, i.e., $\cH(z, w) = \partial_{\zbar}\phi_{a(z, w)}(z)$,  is well-defined for every $(z,w)\in \C\times\C$. Furthermore, by (F2) in Definition~\ref{perhe}, 
$$
|\partial_{\zbar}\phi_a(z) - \partial_{\zbar}\phi_b(z)| \leq \frac{K-1}{K+1}\,|\partial_z \phi_a(z) - \partial_z \phi_b(z)| 
$$
and, 
since $z \mapsto D_z \phi_a(z)$ is continuous,  the inequality holds at every point $z$. Thus, for fixed $z$, we get directly from  \eqref{defH} that
$$
\aligned
|\cH(z, w_1) - \cH(z, w_2)| &= |\partial_{\zbar}\phi_{a(z, w_1)}(z) - \partial_{\zbar}\phi_{a(z, w_2)}(z)| \\ &\leq \frac{K-1}{K+1}\,|\partial_z \phi_{a(z, w_1)}(z) - \partial_z \phi_{a(z, w_2)}(z)|\\
& = \frac{K-1}{K+1}\,|w_1 - w_2|
\endaligned
$$
and $w \mapsto \cH(z, w)$  is $k(z)$-Lipschitz for every $z$, with $\|k\|_\infty\leq\frac{K-1}{K+1}$. Also, $\cH(z, 0) \equiv 0$ by \eqref{defH}, since $\phi_0 \equiv 0$. Hence (H1) holds.

To see (H2), we recall that $\cH$ has been defined in \eqref{compH} as the composition  of two continuous functions, since $z\mapsto a(z,w)$ is continuous, by Lemma~\ref{continuityofa}, and $a\mapsto\partial_{\bar{z}}\phi_a(z)$ is continuous, by assumption. As a continuous function $z\mapsto \cH(z,w)$ is, in particular, measurable. 
 \end{proof}

\subsection{Regularity of $\cH_{\cF}$}\label{regularitysec}

In this final section we conclude the paper by showing how $\cH_{\cF}$ inherits the regularity of $\cF$. Looking at how we defined $\cH_{\cF}$ at \eqref{compH}, it is clear that we first need to understand the smoothness of $(z,w)\mapsto a(z,w)$. 

\begin{lem}\label{moregularityofa}
Let $\cF=\{\phi_a(z)\}_{a\in\C}$ be a non-degenerate field of $K$-quasi\-conformal mappings such that for some $s\in(0,1)$
\begin{equation}\label{linearboundH}
\|D_z\phi_a\|_{C^s(\D_r)}\leq c(s,r)\,|a|, \qquad R < \infty,
\end{equation}  
and  the field satisfies the smoothness condition \eqref{smooth}.
Then
\begin{enumerate}
\item[(a)] $z\mapsto a(z,w) \in C^s_{\loc}(\C)$ with the upper bound for $C^s_{\loc}$-norm that is uniform for $w$ in a compact set and
\item[(b)] $(z, w) \mapsto D_w a(z,w)$ is continuous.
\end{enumerate}
\end{lem}

\begin{proof}
We start by proving $(a)$. Let us denote 
$$G(z,a,w):= \partial_z\phi_a(z)-w.$$
Fix a disk $\D_r\subset \C$ and two points $z_ 0, z_1\in \D_r$. 

Notice that by \eqref{linearboundH}
\begin{equation}\label{Fupperbound}\begin{aligned}
\frac{|G(z_1, a(z_1,w), w)-G(z_0, a(z_1,w), w)|}{|z_1-z_0|^s} & \leq \|D_z\phi_a\|_{C^s(\D_r)}\\ & \leq c(\D_r)\,|a(z_1, w)|.
\end{aligned}
\end{equation}
Since $D_a((\partial_z \phi_a)(z)) \neq 0$ and it is by assumption continuous with respect to $(z, a)$, and $z\mapsto a(z,w)$ is continuous (Lemma~\ref{continuityofa}), we get $$
\inf_{z\in \D_r, |\xi|=1}|D_a((\partial_z\phi_a)(z))\vert_{a(z,w)}\,\xi|  \geq c(w, \D_r) > 0.
$$
Further, as $a(z,w)$ is continuous on $w$ (Lemma~\ref{continuityofa}), the infimum stays bounded for $z_0 ,z_1\in \D_r$ and $w$ in a compact set.

Thus, by the mean-value theorem,
\begin{equation}\label{lowerforinf}
\aligned
&\frac{|G(z_0, a(z_0, w), w)-G(z_0, a(z_1,w), w)|}{|a(z_0, w)-a(z_1, w)|}\\
&\qquad\geq \inf_{z\in\D_r, |\xi|=1}|D_a((\partial_z\phi_a)(z))\vert_{a(z,w)}\,\xi| \geq c(w, \D_r)
\endaligned
\end{equation}
and the lower bound is uniform for $w$ in a compact set.

Then, for each $w\in\C$
$$G(z_1,a(z_1,w), w)=G(z_0,a(z_0,w),w)=0.$$
Hence
\begin{align*}
&\frac{|G(z_1, a(z_1,w), w)-G(z_0, a(z_1,w), w)|}{|z_1-z_0|^s}\\
&\qquad=\frac{|G(z_0, a(z_0, w), w)-G(z_0, a(z_1,w), w)|}{|a(z_0, w)-a(z_1, w)|}\,\frac{|a(z_0, w)-a(z_1, w)|}{|z_0-z_1|^s}.
\end{align*}
The expression has an upper bound by \eqref{Fupperbound}, and combining this with non-degeneracy (as in \eqref{aboundbyw}  we get  $|a(z,w)|\le C(r)|w|$) and \eqref{lowerforinf} gives
$$
\frac{|a(z_0, w)-a(z_1, w)|}{|z_0-z_1|^s}\leq  \frac{c(\D_r)\,|a(z_1,w)|}{c(w,\D_r)} \leq C(w, \D_r)
$$
and $C(w, \D_r)$ is uniform for $w$ in a compact set.

 \medskip

For $(b)$, by assumption
$a\mapsto \partial_z\phi_a(z) \in  C^1(\C)$ while from non-degeneracy we know that $\det D_a(\partial_z\phi_a)(z) \neq 0 $ at every $a$,  $z\in\C$. Then according to the Inverse Function Theorem $w\mapsto a=a(z,w)$, the inverse of  $a\mapsto w=\partial_z\phi_a(z)$, is $C^1$ in the $w$ variable. In fact 
\begin{equation*}
D_w a(z,w)= \bigg(D_a (\partial_z\phi_a)(z)|_{a=a(z,w)}\bigg)^{-1}.
\end{equation*}
Then the smoothness \eqref{smooth} and Lemma \ref{continuityofa} imply the claim.
\end{proof}

Finally, $\cH_{\cF}$ inherits the regularity of $\cF$ in the following way.

\begin{thm}\label{regularity}
Let $\cF=\{\phi_a(z)\}_{a\in\C}$ be a non-degenerate and smooth field of $K$-quasi\-conformal mappings 
in the sense of Definition \ref{nondegdef}, and assume that  \eqref{linearboundH} holds for some $ s\in(0,1)$.

Then there is a unique nonlinear Beltrami equation \eqref{Hqr} such that every mapping of the field $\phi_a$ is a solution to \eqref{Hqr}. The structure function $\cH = \cH_{\cF}$ satisfies (H1), (H2), and the following regularity conditions:
$$
|\cH(z_1, w) - \cH(z_2, w)| \leq c(\D_r, \D_R)|z_1 - z_2|^s, \qquad z_i \in \D_r, w \in \D_R
$$
and $(z, w) \mapsto D_w \cH(z, w)$ is continuous.
\end{thm}

\begin{proof}
In view of Theorem \ref{family} we only need to show the extra smoothness of $\cH$. Let $z, z_1, z_2 \in \D_r$ and $w, w_1, w_2\in\D_R$. From the definition of $\cH$, \eqref{defH},
$$
\aligned
&|\cH(z_1, w) - \cH(z_2, w)| = |\partial_{\zbar}\phi_{a(z_1, w)}(z_1) - \partial_{\zbar}\phi_{a(z_2, w)}(z_2)|\\
&\quad \leq|\partial_{\zbar}\phi_{a(z_1, w)}(z_1) - \partial_{\zbar}\phi_{a(z_2, w)}(z_1)| + |\partial_{\zbar}\phi_{a(z_2, w)}(z_1) - \partial_{\zbar}\phi_{a(z_2, w)}(z_2)|\\
&\quad \leq c(\D_r, \D_R) |a(z_1, w) - a(z_2, w)| + c(s, \D_r)|z_1 - z_2|^s |a(z_2, w)|\\
&\quad\leq c(\D_r, \D_R)|z_1 - z_2|^s,
\endaligned
$$
 where the second to the last inequality follows as $a \mapsto D_z \phi_a \in C^1(\C)$ (and thus in particularly locally Lipschitz) and by \eqref{linearboundH}. In the last inequality we use Lemma~\ref{moregularityofa} $(a)$ and the non-degeneracy as in \eqref{aboundbyw}.
 
 Concerning the continuity of $D_w \cH(z, w)$, the chain rule shows that 
 $$D_w \cH(z, w) = (D_a\partial_{\zbar}\phi_{a})\vert_{a=a(z, w)} \,D_w a(z, w)
$$
is continuous by smoothness \eqref{smooth} and Lemma~\ref{moregularityofa} $(b)$.
\end{proof}

\begin{rem}
As we have discussed several times in the paper, in the absence of regularity,  one can start with a general field of quasiconformal mappings and define an equation $\cH_{\cF}$. Unfortunately if we do not make further assumptions, (H1) and (H2) do not need to hold. As for the ellipticity, by quasiconformality it holds that,  for fixed $a,b$, there is a set 
of measure zero $N_{a,b}$ such that 
$$
|\partial_{\zbar}\phi_a(z) - \partial_{\zbar}\phi_b(z)| \leq k|\partial_z \phi_a(z) - \partial_z \phi_b(z)| \qquad \text{for every $z\in\C\setminus N_{a,b}$}.
$$
Thus, for fixed $w_1,w_2$, the corresponding $\cH_{\cF}$ would be only elliptic up to measure zero depending on $w_1,w_2$. The measurability is also not clear unless one assumes extra conditions.  It would suffice, for example, to prove that $a(z,w)$ is measurable and satisfies the Lusin condition $N^{-1}$ in order that the composition  $\cH(z,w)= \partial_{\zbar}\phi_a(z)\vert_{a=a(z,w)}$ is measurable \cite[page 73]{AIM} or to assume joint measurability properties of the mapping $(z, a) \mapsto D_z\phi_a(z)$ but we will not pursue this issue here.

All in all, it remains an interesting issue what is the minimal regularity to establish the one-to-one correspondence between structure functions $\cH$ and quasiconformal fields $\cF$ in the nonlinear setting. 
\end{rem}

\end{document}